\documentclass{amsart}
\usepackage{adjustbox,lipsum}
\usepackage{amsmath}
\usepackage{amssymb}
\usepackage[shortlabels]{enumitem}
\usepackage{tabu}
\usepackage{tikz-cd}
\usepackage{extarrows}
\usepackage{makecell}
\usepackage[
        colorlinks, citecolor=darkgreen,
        backref,
        pdfauthor={Samir Siksek},
]{hyperref}

\usepackage{comment}
\newcommand{\Aff}{\mathbb{A}}
\newcommand{\nw}{\mathrm{nw}}

\newcommand{\F}{\mathbb{F}}
\newcommand{\KK}{\mathbb{K}}
\newcommand{\PP}{\mathbb{P}}
\newcommand{\Q}{\mathbb{Q}}

\newcommand{\Z}{\mathbb{Z}}
\newcommand{\aaa}{\mathbf{a}}

\newcommand{\cP}{\mathcal{P}}

\newcommand{\fm}{\mathfrak{m}}

\newcommand{\OO}{\mathcal{O}}

\DeclareMathOperator{\Stab}{Stab}
\DeclareMathOperator{\rad}{rad}

\DeclareMathOperator{\divv}{div}

\DeclareMathOperator{\End}{End}
\DeclareMathOperator{\Gal}{Gal}

\DeclareMathOperator{\ord}{ord}

\DeclareMathOperator{\Sel}{Sel}

\renewcommand{\setminus}{-}

\newtheorem{thm}{Theorem}

\newtheorem{lem}[thm]{Lemma}
\newtheorem{conj}[thm]{Conjecture}
\newtheorem{cor}[thm]{Corollary}
\newtheorem{prop}[thm]{Proposition}

\theoremstyle{definition}
\newtheorem{example}[thm]{Example}

\theoremstyle{remark}

\definecolor{darkgreen}{rgb}{0,0.5,0}

\begin{document}

\title[]{
Primitive Algebraic Points on Curves
}

\begin{abstract}
A number field $K$ is primitive if $K$ and $\Q$ are the only subextensions of $K$.
Let $C$ be a curve defined over $\Q$. 
We call an algebraic point $P\in C(\overline{\Q})$ primitive if the number
field $\Q(P)$ is primitive.  We present several sets of sufficient conditions
for a curve $C$ to have finitely many
primitive points of a given degree $d$.
For example, let $C/\Q$ be a hyperelliptic curve of genus $g$,
and let $3 \le d \le g-1$. Suppose that the Jacobian $J$ 
of $C$ is simple. We show that $C$ has only finitely many primitive degree $d$ points,
and in particular it has only finitely many degree $d$ points with Galois group $S_d$
or $A_d$. However, for any even $d \ge 4$,
a hyperelliptic curve $C/\Q$ has infinitely many imprimitive 
degree $d$ points whose Galois group is a subgroup of 
	$S_2 \wr S_{d/2}$.
\end{abstract}

\author{Maleeha Khawaja}

\address{
	School of Mathematics and Statistics\\
	Hicks Building\\
	University of Sheffield\\
	Sheffield S3 7RH\\
	United Kingdom
	}
\email{mkhawaja2@sheffield.ac.uk}

\author{Samir Siksek}

\address{Mathematics Institute\\
    University of Warwick\\
    CV4 7AL \\
    United Kingdom}

\email{s.siksek@warwick.ac.uk}

\date{\today}
\thanks{
Khawaja is supported by an EPSRC studentship from the University of Sheffield (EP/T517835/1).
Siksek is supported by the
EPSRC grant \emph{Moduli of Elliptic curves and Classical Diophantine Problems}
(EP/S031537/1). }

\maketitle

\section{Introduction}

By a \textbf{curve} $C$ over a field $K$ we mean a smooth projective
and geometrically irreducible variety defined over $K$ having 
dimension $1$. 
We say that a curve $C$ defined over a field $K$
is \textbf{hyperelliptic} if the genus of $C$
is at least $2$ and $C$ admits a degree $2$
morphism $C \rightarrow \PP^1$, defined
over $K$, which we shall refer to as the
\textbf{hyperelliptic morphism}.
We say that $C/K$ has \textbf{$K$-gonality} $m$
if $m$ is the least degree of a non-constant morphism 
$\pi : C \rightarrow \PP^1$ defined over $K$.
Thus, for example, a hyperelliptic curve defined over $K$ has
$K$-gonality $2$. 
We say that $C/K$ is \textbf{bielliptic} if its genus is
at least $2$ and it admits a degree $2$ morphism 
$C \rightarrow E$, defined over $K$, where
$E$ is a curve of genus $1$.

Recall that a number field
$K$ is called \textbf{imprimitive}
if there is some subextension $\Q \subsetneq L \subsetneq K$;
if there is no such subextension then $K$ is called
\textbf{primitive}. Now let $C$ be a curve defined over $\Q$.
Let $P \in C$ be an algebraic point; i.e.\ $P \in C(\overline{\Q})$. 
We say $P$ has \textbf{degree $d$}
if the number field $\Q(P)$ has degree $d$.
We say that $P$ is \textbf{primitive}
if the number field $\Q(P)$ is primitive, otherwise
we say that $P$ is $\textbf{imprimitive}$. 
%Following \cite{BELOV},
A degree $d$ point $P \in C(\overline{\Q})$
is called \textbf{$\PP^1$-isolated} if there is no degree $d$
non-constant morphism $\phi : C \rightarrow \PP^1$,
defined over $\Q$, such that $P$ is in the preimage
of an element of $\PP^1(\Q)$. 
The notion of isolated points was introduced
in \cite{BELOV} and has become
important in understanding low degree points on curves,
particularly on modular curves (e.g. \cite{Bourdon2020odd}, \cite{Ejder}).
It is easy to see that  if $d<m$, where $m$ is the $\Q$-gonality of $C$,
then $P$ is $\PP^1$-isolated. A key observation we make in
this paper is that primitive points
are often $\PP^1$-isolated even if the degree is greater than the gonality.
\begin{thm}\label{thm:isolated}
Let $C$ be a curve defined over $\Q$
with genus $g$ and $\Q$-gonality $m \ge 2$. 
Let $d \ge 2$ be an integer satisfying
	\begin{equation}\label{eqn:gencondgon}
		d \ne  m, \qquad 
		d \, < \, 1+\frac{g}{m-1}.
		%g > (m-1)(d-1).
	\end{equation}
%Let $J$ be the Jacobian of $C$. 
%Suppose either of the following hold:
%\begin{enumerate}[(a)]
%\item $J(\Q)$ is finite;
%\item or $d \le g-1$ and $J$ is simple.
%\end{enumerate}
Let $P \in C(\overline{\Q})$ be a degree $d$ point on $C$
that is not $\PP^1$-isolated. Then $\Q(P)$
contains a subfield of index $d^\prime$ satisfying
\[
	1<d^\prime<d, \qquad d^\prime \mid \gcd(d,m).
\]
In particular, the following hold.
\begin{enumerate}[(I)]
\item If $\gcd(d,m)=1$ or $d$ is prime then any degree $d$ point 
$P \in C(\overline{\Q})$ is $\PP^1$-isolated.
\item If $P \in C(\overline{\Q})$ is a primitive
degree $d$ point, then $P$ is $\PP^1$-isolated.
\end{enumerate}
\end{thm}
Under certain additional assumptions on the Jacobian $J$ of the curve $C$,
it is possible to conclude finiteness of primitive degree $d$ points on $C$.
\begin{thm}\label{thm:gonality}
Let $C$ be a curve defined over $\Q$
with genus $g$ and $\Q$-gonality $m \ge 2$. 
Let $d \ge 2$ be an integer satisfying \eqref{eqn:gencondgon}.
%\begin{equation}\label{eqn:gencondgon}
%		d \ne  m, \qquad g > (m-1)(d-1).
%\end{equation}
Let $J$ be the Jacobian of $C$.
Suppose either of the following hold:
\begin{enumerate}[(a)]
\item $J(\Q)$ is finite;
\item or $d \le g-1$, 
and $A(\Q)$ is finite for every abelian subvariety $A/\Q$ of $J$ of dimension $\le d/2$.
\end{enumerate}
Then $C$ has finitely many primitive degree $d$ points. Moreover, if $\gcd(d,m)=1$ or $d$ is prime then
$C$ has finitely many degree $d$ points. 
\end{thm}
Observe that, for $d \le g-1$, assumption (b) is trivially satisfied
if $J$ is simple.
We moreover note that the inequality $d \le g-1$ in (b) follows immediately 
from \eqref{eqn:gencondgon} if $m \ge 3$. 
\begin{example}
 We consider the modular curve $X_{0}(239)$ which has genus $g=20$ and
 $\Q$-gonality $m=6$ (see \cite[Table 3]{najman2023gonality}).
We consider degree $d$ points for $d=2$, $3$, $4$;
we note that \eqref{eqn:gencondgon} is satisfied
for these values of $d$.
 A straightforward computation in \texttt{Magma}, 
which makes use of modular symbols algorithms due to Cremona \cite{Cre} and Stein \cite{Stein},
shows that the Jacobian
        $J_{0}(239)$ of $X_{0}(239)$ factors as
        \[
                J_{0}(239)\sim \mathcal{A}_{3}\times \mathcal{A}_{17},
        \]
        where $\mathcal{A}_{3}$ and $\mathcal{A}_{17}$ are
        simple abelian varieties of dimension $3$ and $17$, respectively.
Moreover $\mathcal{A}_{17}$ has analytic rank $0$,
and $\mathcal{A}_{3}$ has positive analytic rank. Assuming the Birch and Swinnerton--Dyer
conjecture,
$J_0(239)(\Q)$ is infinite, and so hypothesis (a) of Theorem~\ref{thm:gonality}
is not satisfied.
However, $J_0(239)$,
clearly satisfies hypothesis (b) for $d=2$, $3$, $4$.
By Theorem~\ref{thm:gonality},
we conclude that $X_{0}(239)$ has finitely many quadratic, cubic
and primitive quartic points.
\end{example}

We point out the following theorem 
\cite[Proposition 2.3]{DerickxSutherland}
which gives a stronger conclusion than Theorem~\ref{thm:gonality}
under different assumptions.
\begin{thm}[Derickx and Sutherland]
Let $C$ be a curve defined over $\Q$
with genus $g$ and $\Q$-gonality $m \ge 2$. 
Let $J$ be the Jacobian of $C$ and suppose $J(\Q)$
is finite. Suppose $d<m$. Then $C$ has finitely
many degree $d$ points.
\end{thm}

The following corollary to Theorem~\ref{thm:gonality} is obtained by restricting
Theorem~\ref{thm:gonality} and its proof to the hyperelliptic case.
\begin{cor}\label{cor:hyp}
Let $C$ be a hyperelliptic curve defined over $\Q$ 
with genus $g$. 
Let $J$ be the Jacobian
of $C$ and let $d$ be a positive integer.
Suppose either of the following hold:
\begin{enumerate}[(a)]
\item $3 \le d \le g$ and $J(\Q)$ is finite;
\item or $3 \le d \le g-1$,
and $A(\Q)$ is finite for every abelian subvariety $A/\Q$ of $J$ of dimension $\le d/2$.
\end{enumerate}
Then $C$
has finitely many primitive degree $d$ points.
More precisely, the following hold.
\begin{enumerate}[(i)]
\item If $d$ is odd, then $C$
has finitely many degree $d$ points.
\item If $d$ is even, then
for all but finitely many
degree $d$ points $P$ on $C$,
%the image $\pi(P) \in \PP^1(\overline{\Q})$
%has degree $d/2$, and therefore the degree $d$
the field $\Q(P)$ contains a subfield of index $2$.
%the degree $d/2$ subfield $\Q(\pi(P))$.
\end{enumerate}
\end{cor}
We note that Gunther and Morrow \cite[Proposition 2.6]{GuntherMorrow}
prove that conclusions (i) and (ii) hold for $100\%$ of genus $g$ hyperelliptic curves over $\Q$
with a rational Weierstrass point, provided only that $d \le g-1$, though their proof
is rather different.

In contrast to Corollary~\ref{cor:hyp} we note the following.
\begin{lem}\label{lem:hypimprim}
Let $C$ be a hyperelliptic curve defined over $\Q$.
Let $d \ge 4$ be an even integer. Then $C$ 
has infinitely many imprimitive degree $d$ points.
\end{lem}
\begin{proof}
We may suppose $C$ has an affine model
\begin{equation}\label{eqn:hyper}
	C \; : \; Y^2 = F(X)
\end{equation}
where $F \in \Q[X]$ is a squarefree polynomial.
Let $L$ be any number field of degree $d/2$
and choose $\theta \in L$ such that $L=\Q(\theta)$.
By Faltings' theorem, $C(L)$ is finite. Thus 
we may choose some $a \in \Q$ such that
$F(\theta+a)$ is a non-square in $L$. 
Let $P=(\theta+a,\sqrt{F(\theta+a)})$. This is a degree $d$ point on $C$,
and is imprimitive as $\Q(P)$ contains the index $2$
subfield $L$.
\end{proof}

Let $C$ be a curve defined over $\Q$.
Let $P \in C$ be an algebraic point; i.e.\ $P \in C(\overline{\Q})$. 
We define the \textbf{Galois group of $P$}
to be the Galois group of the Galois closure
of $\Q(P)/\Q$. A degree $d \ge 4$ point whose
Galois group is $S_d$ or $A_d$ is primitive (Lemma~\ref{lem:pointSdprim}).
Thus if $d$ is an even integer and if $C$, $d$ satisfy the 
hypotheses of Corollary~\ref{cor:hyp} then $C$ has only finitely
many degree $d$ points with Galois group $S_d$ or $A_d$.
However, it follows from the proof of Lemma~\ref{lem:hypimprim}
that $C$ has infinitely many degree $d$ points
with Galois group contained in $S_{2} \wr S_{d/2}$.
In a separate paper 
\cite{KhawajaSiksek2} we explore Galois groups of 
algebraic points in more detail. For now, we content
ourselves with the following result.
\begin{thm}\label{thm:trivialjac}
Let $C$ be a hyperelliptic curve defined over $\Q$ with genus $2$ or $3$.
Let $J$ be the Jacobian of $C$ and suppose that $J(\Q)$ is
trivial.  Then $C$ has no quartic points with Galois group $S_4$
or $A_4$.
However, $C$ has infinitely many quartic points with
Galois group contained in $D_4$.
\end{thm}

The above results exploit the gonality map $C \rightarrow \PP^1$ to 
make deductions about low degree primitive points.
However, the existence of a low degree map $C \rightarrow C^\prime$ with $C^\prime$ of positive genus
also makes it more likely for low degree primitive algebraic points on $C$ to
be $\PP^1$-isolated, as illustrated by the following theorem.
\begin{thm}\label{thm:relativeisolated}
Let $\pi : C \rightarrow C^\prime$ be a morphism of curves defined over $\Q$ of degree $m \ge 2$.
Write $g$, $g^\prime$ for the genera of $C$, $C^\prime$
respectively, and suppose $g^\prime \ge 1$. 
Let $d \ge 2$ be an integer satisfying
\begin{equation}\label{eqn:gencondrel}
	d \, < \, 1+\frac{g-m g^\prime}{m-1}.
	%g \; > \; m g^\prime+(m-1)(d-1).
\end{equation}
Let $P \in C(\overline{\Q})$ be a degree $d$ point on $C$
that is not $\PP^1$-isolated. Then $\Q(P)$
contains a subfield of index $d^\prime$ satisfying
\[
	1<d^\prime<d, \qquad d^\prime \mid \gcd(d,m).
\]
In particular, the following hold.
\begin{enumerate}[(I)]
\item If $\gcd(d,m)=1$ or $d$ is prime then any degree $d$ point 
$P \in C(\overline{\Q})$ is $\PP^1$-isolated.
\item If $P \in C(\overline{\Q})$ is a primitive
degree $d$ point, then $P$ is $\PP^1$-isolated.
\end{enumerate}
\end{thm}

\begin{thm}\label{thm:relative}
Let $\pi : C \rightarrow C^\prime$ be a morphism of curves defined over $\Q$ of degree $m \ge 2$.
Write $g$, $g^\prime$ for the genera of $C$, $C^\prime$
respectively, and suppose $g^\prime \ge 1$. 
Let $d \ge 2$ be an integer satisfying \eqref{eqn:gencondrel}.
Write $J$ for the Jacobian of $C$ and suppose $J(\Q)$ is finite.
Then $C$ has finitely many primitive degree $d$ points. Moreover, if $\gcd(d,m)=1$ or $d$ is prime then
$C$ has finitely many degree $d$ points.
\end{thm}
\begin{example}\label{ex:X1(45)}
We illustrate Theorem~\ref{thm:relative} by giving an example.
Consider the modular curve $X_1(45)$. 
The \texttt{LMFDB} \cite{lmfdb} gives the following information:
\begin{enumerate}[(a)]
\item $X_1(45)$ has genus $41$;
\item $X_1(45)$ is a degree $3$ cover of a genus $9$ curve (the latter has 
\texttt{LMFDB} label \texttt{45.576.9-45.a.4.1});
\item $J_1(45)$ has analytic rank $0$.
\end{enumerate}
Write $J=J_1(45)$.
As $J$ has analytic rank $0$, a theorem of Kato \cite[Corollary 14.3]{Kato} 
implies that the Mordell--Weil group $J(\Q)$ is finite. Thus, by
Theorem~\ref{thm:relative}, the curve
$X_1(45)$ has only finitely many degree $d$
points for $d=2$, $3$, $4$, $5$, $7$,
and only finitely many primitive degree $d$ points for $d=6$. 
We point out that the $\Q$-gonality of $X_1(45)$ appears to be currently
unknown; according to the \texttt{LMFDB} it belongs to the interval $9 \le
\gamma \le 18$.  
\end{example}

Applying Theorem~\ref{thm:relative} and its proof to bielliptic curves
gives the following.
\begin{cor}\label{cor:bielliptic}
Let $C$ be a bielliptic curve defined over $\Q$ with genus $g$.
%and let $\pi : C \rightarrow E$ be the corresponding
%degree $2$ morphism to a genus $1$ curve.
Let $J$ be the Jacobian of $C$ and suppose $J(\Q)$ is finite.
Let $2\leq d\leq g-2$.
Then $C$ has finitely many primitive degree $d$ points.
More precisely, the following hold.
\begin{enumerate}[(i)]
\item If $d=2$ or $d$ is odd, then $C$ has finitely many degree $d$ points.
\item If $d \ge 4$ and even, then for
all but finitely many
degree $d$ points $P$ on $C$,
%the image $\pi(P) \in E(\overline{\Q})$
%has degree $d/2$, and therefore the degree $d$
the field $\Q(P)$ contains
a subfield of index $2$.
\end{enumerate}
\end{cor}

The above theorems are concerned with finiteness criteria
for low degree primitive points.
In the opposite direction, we show that if the degree $d$ is sufficiently large compared to the genus,
then there are infinitely many primitive degree $d$ points, provided there is at least one such point.
\begin{thm}\label{thm:infprimitive}
Let $C/\Q$ be a curve.
Let $d \ge g+1$ where $g$ is the genus of $C$.
Suppose there exists a primitive
degree $d$ point on $C$. 
Then there are infinitely many primitive
degree $d$ points on $C$.
\end{thm}

The paper is organized as follows. In Section~\ref{sec:prooftrivjac}
we review some standard results regarding Riemann--Roch spaces
and we use these to prove Theorem~\ref{thm:trivialjac}.
In Section~\ref{sec:CS} we recall the Castelnuovo--Severi
theorem and use it
to give proofs of Theorems~\ref{thm:isolated}
and~\ref{thm:relativeisolated}.
In Section~\ref{sec:primDiv}
we give criteria for when 
a complete linear series does not contain
any primitive divisors (a primitive divisor
is simply the Galois orbit of a primitive algebraic
point). 
Let $C$ satisfy the hypotheses of either Theorem~\ref{thm:gonality}
or Theorem~\ref{thm:relative},
and write $C^{(d)}$ for the $d$-th
symmetric power of $C$. 
In Section~\ref{sec:decomp},
using a powerful theorem of Faltings, 
we show that $C^{(d)}(\Q)$ may be decomposed as a finite
union of complete linear series. 
In Section~\ref{sec:proofs}
we prove Theorems~\ref{thm:gonality} and~\ref{thm:relative}
and their corollaries, making use of the results
developed in previous sections.
In Section~\ref{sec:infprimitive} we recall the relationship
between primitive extensions and primitive group actions,
and we use this together with Hilbert's irreducibility
theorem to give a proof of Theorem~\ref{thm:infprimitive}.
Finally, in Sections~\ref{sec:modular1} and~\ref{sec:modular2}
we give consequences of Theorems~\ref{thm:gonality}
and~\ref{thm:relative} (and their proofs) for the modular curves
$X_1(N)$ and $X_0(N)$ for certain small $N$.

We are grateful to Nils Bruin, Victor Flynn, Samuel Le Fourn, David Loeffler, Filip Najman,
Petar Orli\'{c},
Damiano Testa and Bianca Viray for helpful discussions,
and to the referees for suggesting many improvements.

\section{Proof of Theorem~\ref{thm:trivialjac}}\label{sec:prooftrivjac}
Let $C$ be a curve defined
over $\Q$. When we speak of divisors on $C$ we in fact mean rational divisors:
a \textbf{divisor} on $C/\Q$ is a finite formal integral linear combination $D=\sum a_i P_i$ of algebraic points $P_i$ that is stable
under the action of $\Gal(\overline{\Q}/\Q)$. We call this divisor \textbf{effective} and write
$D \ge 0$ if and only if $a_i \ge 0$ for all $i$.  An \textbf{irreducible divisor} is an effective divisor that 
cannot be written as the sum of two non-zero effective divisors. Thus an effective degree $d$ divisor $D$
is irreducible if and only if there is a degree $d$ point $P \in C(\overline{\Q})$ such that
$D=P_1+P_2+\cdots+P_d$ where $\{P_1,\dotsc,P_d\}$ is the Galois orbit of $P$. We say that $D$ is \textbf{the 
irreducible divisor corresponding} to $P$.

For a divisor $D$ on $C$ we denote by
$L(D)$ the corresponding \textbf{Riemann--Roch space} defined by
\[
L(D) \; = \; \{0\} \cup \{ f \in \Q(C)^\times \; : \; \divv(f)+D \ge 0\},
\]
and we let $\ell(D)=\dim L(D)$. 
We shall make frequent use of the Riemann--Roch theorem \cite[page 13]{ACGH} which 
asserts that
\begin{equation}\label{eqn:RR}
	\ell(D)-\ell(K_C-D) \; = \; \deg(D)-g+1;
\end{equation}
here $K_C$ is any canonical divisor on $C$, and $g$ is the genus of $C$.
Recall that $\deg(K_C)=2g-2$. Therefore, if $\deg(D) \ge 2g-1$ then
$K_C-D$ has negative degree and cannot be linearly equivalent to an
effective divisor. In that case $\ell(K_C-D)=0$.

We shall also require 
Clifford's theorem \cite[Theorem IV.5.4]{Hartshorne} on special
effective divisors.
Recall that
an effective divisor $D$ is \textbf{special} if $\ell(K_C-D)>0$.
\begin{thm}[Clifford]\label{thm:Clifford}
Let $D$ be an effective special divisor on a curve $C$. Then
\[
	\ell(D) \; \le \; \frac{\deg(D)}{2} \, + \, 1.
\]
Moreover, equality occurs if and only if $D=0$, or $D$ is a canonical divisor,
or $C$ is hyperelliptic and $D$ is a multiple of a hyperelliptic divisor.
\end{thm}
Recall that a hyperelliptic curve $C$ is equipped
with a degree $2$ morphism $\pi : C \rightarrow \PP^1$;
a \textbf{hyperelliptic divisor} on $C$ is $\pi^*(\alpha)$ for any $\alpha 
\in \PP^1(\Q)$.

\begin{proof}[Proof of Theorem~\ref{thm:trivialjac}]
Let $C$ be as in the statement of Theorem~\ref{thm:trivialjac}.
We may suppose $C$ has an affine model as in \eqref{eqn:hyper}
where $F \in \Q[X]$ is a squarefree polynomial of degree $2g+1$
or $2g+2$. It follows from Lemma~\ref{lem:hypimprim}
and its proof that $C$ has infinitely many quartic points with
Galois group contained in $D_4$.
To complete the proof it is enough to show that there
are no quartic points on $C$ with Galois group $S_4$
or $A_4$.

If $\deg(F)=2g+1$ we let $\infty$ be the single point
at infinity on this model, and write $D_0=4\infty$.
If $\deg(F)=2g+2$ we let $\infty_+$ and $\infty_{-}$ be the two points at infinity,
and write $D_{0}=2\infty_{+}+2\infty_{-}$. 
In either case $D_0$ is twice a hyperelliptic divisor.

Let $P$ be a degree $4$ point on $C$, 
and let $D$ be the corresponding irreducible divisor. 
	Since $J(\Q)$ is trivial, $ D-D_{0} \, \sim \, 0$ where $\sim$ denotes linear equivalence on $C$.
That is,
\[
        D=D_{0}+\divv(f),
\]
where $f\in L(D_{0})$. 
We claim that $1,X,X^2$ is a $\Q$-basis of $L(D_0)$. Let us first assume our claim and use it to complete the proof.
Thus $f=a_0+a_1 X + a_2 X^2$ for some $a_0$, $a_1$, $a_2 \in \Q$. Moreover,
$f$ is non-constant as $D \ne D_0$. Now $P$ is a zero of $f$. Hence $X(P)$ satisfies
the non-constant polynomial $a_0+a_1 U+a_2 U^2 \in \Q[U]$. Since $\Q(P)=\Q(X(P),Y(P))$
is a quartic field, and $Y(P)^2=F(X(P))$, we see that $\Q(X(P))$
is quadratic and contained in the quartic field $\Q(P)$. Therefore the
Galois group of $P$ is neither $S_4$ nor $A_4$.

It remains to prove our claim. Note that $X$ has a double pole
	at infinity and no other poles if $\deg(F)=2g+1$;
	and also
$X$ has a simple pole at $\infty_+$ and $\infty_{-}$, and has no other poles
if $\deg(F)=2g+2$. Therefore, 
	$1$, $X$, $X^2$ belong to $L(D_{0})$, and so $\ell(D_0) \ge 3$.
	It is enough to show that $\ell(D_0)=3$.
We now make use of our assumption that $g=2$ or $3$. If $g=2$ then
	Riemann--Roch immediately gives $\ell(D_0)=3$.
	Suppose $g=3$. Then Riemann--Roch tells us that $D_0$ is special,
	and since it is twice a hyperelliptic divisor, Clifford's theorem
	gives the equality $\ell(D_0)=3$.
\end{proof}

\section{Proofs of Theorems~\ref{thm:isolated}
and~\ref{thm:relativeisolated}}
\label{sec:CS}
\begin{comment}
As noted previously, if $C$ and $d$ satisfy the hypotheses
of either Theorem~\ref{thm:gonality} or~\ref{thm:relative}
then there is a finite decomposition of
$C^{(d)}(\Q)$ as in \eqref{eqn:decomp}.
If $\lvert D_i \rvert$ is infinite,
then  $\ell(D_i) \ge 2$, and provided $\lvert D_i \rvert$
contains an irreducible divisor, Lemma~\ref{lem:func} gives 
a degree $d$ morphism $f: C \rightarrow \PP^1$.
Note that in the hypotheses of Theorems~\ref{thm:gonality}
and~\ref{thm:relative}, the curve $C$ already comes with a 
degree $m$ morphism $\pi : C \rightarrow C^\prime$ (where $C^\prime=\PP^1$
in Theorem~\ref{thm:gonality}). The Castelnuovo--Severi theorem 
allows us to deduce that $f$ and $\pi$ must factor through 
a common morphism, and therefore gives information about $\lvert D_i \rvert$.
We note that under these hypotheses, we have $d<g$.
Although the converse of the aforementioned result of Ejder 
is not generally true, we prove that any primitive degree $d$ point on $C$ is $\PP^1$-isolated.
\end{comment}
%
%We will require some fairly standard results
%concerning linear series of divisors on hyperelliptic curves.
%For the reader's convenience we shall prove these
%from scratch.

We start by recalling the classical Castelnuovo--Severi theorem.
\begin{thm}[Castelnuovo--Severi theorem]
Let $k$ be a perfect field, and let $X$, $Y$, $Z$
be curves over $k$. Denote the genera of these
curves by $g(X)$, $g(Y)$ and $g(Z)$
respectively. Let $\pi_Y : X \rightarrow Y$
and $\pi_Z: X \rightarrow Z$
be non-constant morphisms defined over $k$,
having degrees $m$ and $n$ respectively.
Suppose
\begin{equation}\label{eqn:assumption}
	g(X) \; > \;
	m \cdot g(Y) + n \cdot g(Z) +(m-1)(n-1).
\end{equation}
Then there is a curve $X^\prime$ defined over $k$, and a morphism
$X \rightarrow X^\prime$ 
also defined over $k$
and of degree $>1$ through which both $\pi_Y$ and $\pi_Z$
factor. 
\end{thm}
\begin{proof}
We are unable to find a reference that gives the precise statement
that we need. Indeed the theorem is most 
often given in the context of complex Riemann surfaces (for example \cite{Accola}),
or the field of definition of the morphism $X \rightarrow X^\prime$
is not mentioned (for example \cite[Corollary]{Kani}). We therefore give an explanation of how
the version above follows straightforwardly 
from a proof due to Mattuck and Tate \cite{Mattuck_Tate}
of the Castelnuovo--Severi inequality.
%The version we give here is very slightly stronger
%than other versions found in the literature, e.g.
%\cite[Theorem 3.11.3]{Stichtenoth}, in that we conclude
%that $X^\prime$ and $X \rightarrow X^\prime$ are defined over $k$,
%and not its algebraic closure. We give a justification for this.

Let $S=Y \times Z$. Given two divisors $D$, $D^\prime$ on the surface $S$,
we denote their intersection number by $D \cdot D^\prime$.
Let $D$ be the image of $X$ on  $S$
under the morphism $(\pi_Y,\pi_Z) : X \rightarrow S$,
and write $h : X \rightarrow D$ for the induced map.
Clearly the curve $D$ and the map $h$
are defined over $k$. Let $u: D \rightarrow Y$
and $v : D \rightarrow Z$ be the restrictions
of the projections  $Y \times Z \rightarrow Y$ and $Y \times Z \rightarrow Z$
to $D$.
Then $\pi_Y=u \circ h$ and $\pi_Z=v \circ h$.
We claim that $\deg(h) > 1$. The theorem follows from our claim
on letting $X^\prime$ be the normalization of $D$. 

We prove our
claim by contradiction. Suppose $\deg(h)=1$. The $h$ is a birational map,
and $g(D)=g(X)$ (where $g(D)$ denotes the geometric genus of $D$).
Mattuck and Tate \cite[page 296]{Mattuck_Tate}
show that
\begin{equation}\label{eqn:Mattuck_Tate}
	D \cdot D \; \le \; 2 mn, \qquad D \cdot K_S \; = \; (2g(Y)-2) m+
	(2 g(Z)-2)n,
\end{equation}
	where $K_S$ denotes any canonical divisor on $S$ (the inequality
	on the left is itself referred to as the Castelnuovo--Severi inequality). By the adjunction formula for surfaces
\cite[Exercise V.1.3]{Hartshorne}  we have
\[
	2 p_a(D)-2 \; = \; D \cdot (D+K_S),
	%\; \le \; 2mn+(2g(Y)-2)m+(2g(Z)-2)n,
\]
where $p_a(D)$ is the arithmetic genus of $D$. Thus, from \eqref{eqn:Mattuck_Tate},
\[
	p_a(D) \; \le \; g(Y) m + g(Z) n + (m-1)(n-1).
\]
However, since the geometric genus is bounded by the arithmetic genus,
\[
	g(X) \; = \; g(D)  \; \le \; p_a(D) \; \le \; g(Y) m + g(Z) n + (m-1)(n-1).
\]
contradicting assumption \eqref{eqn:assumption}.
This establishes our claim.
\end{proof}

\begin{comment}
We start with some easy consequences of Castelnuovo--Severi.
\begin{cor}\label{cor:main1}
Let $k$ be a perfect field.
Let $C$ be a hyperelliptic curve defined over $k$
of genus $g$,
and let $\pi : C \rightarrow \PP^1$ be
the associated hyperelliptic morphism. 
Let $2 \le d \le g$
and suppose $f : C \rightarrow \PP^1$ is a degree $d$
morphism defined over $k$. Then $d$ is even and
$f$ factors via $\pi$.
\end{cor}
\begin{proof}
If $f$ does not factor via $\pi$ then
we may apply the Castelnuovo--Severi
theorem to deduce that $g \le d-1$,
giving a contradiction.
\end{proof}
Note that in the case $d=2$, the corollary is simply
asserting the well-known fact that
a hyperelliptic curve has a unique hyperelliptic
morphism, up to composition by an automorphism of $\PP^1$.

\begin{cor}\label{cor:genge4}
Let $k$ be a perfect field, and let $C$ be a hyperelliptic
curve defined over $k$ of genus $g \ge 4$.
Then $C$ is not bielliptic.
\end{cor}
\begin{proof}
Suppose $C$ is hyperelliptic and bielliptic.
Then there are degree $2$ morphisms
$\pi : C \rightarrow \PP^1$ and $b : C \rightarrow E$
where $E$ has genus $1$. Suppose $g \ge 4$.
By Castelnuovo--Severi, $\pi$ and $b$
must factor through each other and so $E$ is isomorphic
to $\PP^1$ giving a contradiction.
\end{proof}
If $g=2$ or $3$, then $C$ can be both hyperelliptic and bielliptic,
as evident from the following two examples
\[
C_1 \; : \; Y^2=X^6+1, \qquad C_2 \; : \; Y^2=X^8+1, 
\]
which respectively have obvious morphisms to the elliptic curves
$Y^2=X^3+1$ and $Y^2=X^4+1$.
\end{comment}

\begin{proof}[Proof of Theorem~\ref{thm:isolated}]
Let $P$ be a degree $d$ point on $C$ and suppose $P$ 
is not $\PP^1$-isolated. Thus there is a degree $d$
map $f : C \rightarrow \PP^1$, defined over $\Q$,
such that $f(P) = \alpha \in \PP^1(\Q)$.
Observe that $f^*(\alpha)$ consists precisely
of the Galois orbit of $P$.
Let $m \ge 2$ be the $\Q$-gonality of $C$; thus
in particular, there is a morphism $\pi : C \rightarrow \PP^1$
defined over $\Q$, of minimal degree $m \ge 2$.
As $m$ and $d$ satisfy \eqref{eqn:gencondgon},
the Castelnuovo--Severi theorem immediately implies that 
the morphisms $f$, $\pi$ must simultaneously factor through
a common non-trivial morphism of curves. 
We obtain a commutative diagram of non-constant morphisms of curves
\begin{equation}\label{eqn:diagram1}
\begin{tikzcd}
	& C \arrow[ld, "f" above] \arrow[d, "h"] \arrow[rd, "\pi"] & \\
	\PP^1 & Y \arrow[l, "u"] \arrow[r, "v" below] & \PP^1
\end{tikzcd}
\end{equation}
defined over $\Q$, where $\deg(h)>1$. 
Write $d^\prime=\deg(h)$. Then $d^\prime$ divides
both $d=\deg(f)$ and $m=\deg(\pi)$, so $d^\prime \mid \gcd(d,m)$. 
If $d^\prime=d$, then $d \mid m$. However, $m \le d$
by the minimality of $m$. Therefore $m=d$ contradicting
\eqref{eqn:gencondgon}. We deduce that $d^\prime<d$.
%Let $Q=h(P)$, so that $\alpha=u(Q)$.
	%Since $\deg(P)=\deg(f)=d$, we see that $\deg(Q)=\deg(u)=d/d^\prime$.
Let $Q=h(P) \in Y$. Since $\Gal(\overline{\Q}/\Q)$
acts transitively on $f^*(\alpha) \ni P$, it acts transitively
on $u^*(\alpha) \ni Q$. Hence $Q$ has degree $\deg(u)=d/d^\prime$
and $\Q(Q) \subseteq \Q(P)$. 
	Thus the field $\Q(P)$ of degree $d$ contains the
	subfield $\Q(Q)$ of index $d^\prime$. The theorem follows.
%We conclude
%that if $\gcd(d,m)=1$ or $d$ is prime then $\lvert D^\prime\rvert$
%contains no irreducible divisors.
%
%We now suppose $\gcd(d,m)>1$.
%If $\deg(u)=1$, then $\pi$ factors
%via $f$, and so $d \mid m$. However, as $m$ is the $\Q$-gonality 
%and $d \ne m$ we have a contradiction. Thus $\deg(u)>1$.
%so $P$ is imprimitive.
\end{proof}

\begin{comment}
\begin{cor}\label{cor:relative}
Let $\pi : C \rightarrow C^\prime$ be a non-constant morphism of curves defined over $\Q$.
Let $m$ be the degree of $\pi$ and write $g$, $g^\prime$ for the genera of $C$, $C^\prime$
respectively, and suppose $g^\prime \ge 1$. 
Let $d$ be a positive integer satisfying \eqref{eqn:gencondrel}.
Suppose $P\in C(\overline{\Q})$ is a degree $d$ point. 
If $P$ is primitive, then $P$ is $\PP^1$-isolated.
Moreover if $\gcd(m,d)=1$ or $d$ is prime, then $P$ is $\PP^1$-isolated.
\end{cor}
\end{comment}
\begin{proof}[Proof of Theorem~\ref{thm:relativeisolated}]
This proof is similar to the proof of Theorem~\ref{thm:isolated}.
As before, let $P$ be a 
degree $d$ point on $C$ and suppose that $P$ is not $\PP^1$-isolated.
Thus there is a degree $d$
map $f : C \rightarrow \PP^1$, defined over $\Q$,
such that $f(P) = \alpha \in \PP^1(\Q)$.

%and let $D$ be the corresponding degree $d$ divisor.
%Suppose $D \in \lvert D^\prime\rvert$. Then $\ell(D)=\ell(D^\prime) \ge 2$.
%By Lemma~\ref{lem:func} there is 
%a degree $d$ morphism $f : C \rightarrow \PP^1$ defined over $\Q$
%with $f^*(\infty)=D$.  
%Again, we shall prove that $P$ is imprimitive.

From assumption \eqref{eqn:gencondrel},
the Castelnuovo--Severi theorem gives
a commutative diagram of non-constant morphisms of curves
\begin{equation}\label{eqn:diagram2}
\begin{tikzcd}
	& C \arrow[ld, "f" above] \arrow[d, "h"] \arrow[rd, "\pi"] & \\
	\PP^1 & Y \arrow[l, "u"] \arrow[r, "v" below] & C^\prime 
\end{tikzcd}
\end{equation}
defined over $\Q$, where $\deg(h)>1$.
As before we let $d^\prime=\deg(h)$.
If $\deg(u)=1$ then $Y \cong \PP^1$ 
which contradicts the existence of 
a non-constant morphism $v$ from $Y$
to the curve $C^\prime$ having genus $g^\prime \ge 1$.
Thus $\deg(u)>1$ and $d^\prime=d/\deg(u)<d$.
The rest of the proof is similar to the 
proof of Theorem~\ref{thm:isolated}.
%As before we conclude 
%	that if $\gcd(d,m)=1$ then $\lvert D^\prime\rvert$
%	also contains no irreducible divisors. 
%	The rest of the proof is as before.
\end{proof}

\begin{comment}
We shall strengthen the conclusion of the above corollary when
$C$ is hyperelliptic.
\begin{cor}\label{cor:gonalityhyp}
Let $C$ be a hyperelliptic curve defined over $\Q$ with genus $g$ and 
$\pi : C \rightarrow \PP^1$ be the hyperelliptic morphism.
Let $d$ be an integer satisfying $3 \le d \le g$. 
Then any degree $d$ point on $C$ is $\PP^1$-isolated.

Let $D^\prime$ be an effective degree $d$ divisor on $C$
with $\ell(D^\prime) \ge 2$.
Let $P$ be a degree $d$ point such that the
corresponding irreducible divisor $D$ belongs to $\lvert D^\prime \rvert$.
Then $d$ is even, $\pi(P) \in \PP^1(\overline{\Q})$
has degree $d/2$, and so $\Q(P)$ contains the
subfield $\Q(\pi(P))$ of degree $d/2$.
\end{cor}
\begin{proof}
This can be deduced from the proof of Corollary~\ref{cor:gonality},
but it is less confusing to prove the corollary from scratch.
Suppose $\ell(D^\prime)\ge 2$.
Let $P$ be a degree $d$ point and let $D$
be the corresponding irreducible divisor.
Suppose $D \in \lvert D^\prime \rvert$.
By Lemma~\ref{lem:func} there
is a degree $d$ morphism $f: C \rightarrow \PP^1$,
defined over $\Q$, satisfying $f^*(\infty)=D$.
Again, we shall prove that $P$ is imprimitive.

Applying Corollary~\ref{cor:main1}
we deduce that $d$ is even and 
that $f$ factors via $\pi$. Hence
$f=u \circ \pi$ where $u: \PP^1 \rightarrow \PP^1$ has degree $d/2$.
	Let $Q=\pi(P)$.
	As $\Gal(\overline{\Q}/\Q)$ acts transitively on
	$f^*(\infty) \ni P$, it acts transitively on $u^*(\infty) \ni Q$.
	Hence $Q$ has degree $d/2$ as required.
\end{proof}
\end{comment}

\section{Primitive Divisors in Complete Linear Series}
\label{sec:primDiv}
Let $C$ be a curve over $\Q$.
For an effective divisor $D$ on $C$
the notation $\lvert D \rvert$ denotes the \textbf{complete linear series
containing $D$}:
\[
	\lvert D \rvert = \{ D+\divv(f) : f \in L(D)  \};
\]
this is precisely the set of effective divisors linearly equivalent to $D$.
Observe that for effective divisors $D$, $D^\prime$,
we have $D \sim D^\prime$ if and only if $\lvert D \rvert=\lvert D^\prime \rvert$.
Recall that $\lvert D \rvert \cong \PP^{\ell(D)-1}(\Q)$.
Note that $\ell(D) \ge 1$ since $\Q \subseteq L(D)$ for any effective divisor $D$.
In particular, $\lvert D \rvert=\{D\}$
if and only if $\ell(D)=1$.

The following well-known lemma highlights the 
relationship
between a point being not $\PP^1$-isolated
and the complete linear series of its irreducible
divisor having positive dimension.
\begin{lem}\label{lem:func}
Let $C$ be a curve defined over $\Q$, and let $d \ge 1$.
Let $D$ be an irreducible degree $d$ divisor.
The following are equivalent.
\begin{enumerate}[(a)]
\item $\dim \lvert D \rvert \ge 1$.
\item $\ell(D) \ge 2$.
\item There is a degree $d$ morphism $f : C \rightarrow \PP^1$,
defined over $\Q$, such that $f^*(\infty)=D$.
\item Any $P \in D$ is not $\PP^1$-isolated.
\end{enumerate}
%and suppose that $\ell(D) \ge 2$. 
%Let $D$, $D^\prime$ be effective degree $d$ divisors on $C$,
%such that $D \sim D^\prime$ and $D \ne D^\prime$. 
%Suppose that at least one of $D$ and $D^\prime$ is irreducible.
\end{lem}
\begin{proof}
Recall that $\lvert D \rvert \cong \PP^{\ell(D)-1}(\Q)$.
Thus $\dim \lvert D \rvert=\ell(D)-1$
and therefore (a) and (b) are equivalent.

Suppose $\ell(D) \ge 2$. Then there is some non-constant $f \in L(D)$.
Thus $f \in \Q(C)^\times$ satisfies $\divv(f)+D \ge 0$. Write $\divv_\infty(f)$
for the divisor of poles of $f$. As $f$ is non-constant,
$0<\divv_\infty(f)$. As $f$ is defined over $\Q$, we have $\divv_\infty(f)$
is a rational divisor. Moreover, $\divv_\infty(f) \le D$ since $\divv(f)+D \ge 0$.
Since $D$ is irreducible, $\divv_\infty(f)=D$.
Now we consider $f$ as a morphism $C \rightarrow \PP^1$ defined over $\Q$.
	Then $f^*(\infty)=D$. As $D$ has degree $d$ so does $f$. Thus (b)
	implies (c).

Conversely, suppose (c). Then $f \in L(D)$ and non-constant giving (b).
	
	It is clear that (c) implies (d). To complete the proof
	we suppose (d) and prove (c). Thus there is
	a morphism $f : C \rightarrow \PP^1$ of degree $d$
defined over $\Q$ with $f(P)=\alpha \in \PP^1(\Q)$.
Composing $f$ with an automorphism of $\PP^1$ we may
suppose $\alpha=\infty$. Now $P \in \divv_\infty(f)$,
and so $D \le \divv_\infty(f)$ as $\divv_\infty(f)$ is stable under the action
of Galois. Thus $D=\divv_\infty(f)$ as both divisors have degree $d$.
This completes the proof.
\end{proof}

We will call an irreducible divisor \textbf{primitive}
if it is the Galois orbit of a primitive point.
\begin{cor}\label{cor:isolated}
Let $C$ be a curve defined over $\Q$ with genus $g$ and $\Q$-gonality $m \ge 2$.
Let $d \ge 2$ be an integer satisfying \eqref{eqn:gencondgon}. 
Let $D^\prime$ be an effective degree $d$ divisor 
on $C$ with $\ell(D^\prime) \ge 2$.
Then $\lvert D^\prime \rvert$ contains no 
primitive degree $d$ divisors. 
Moreover, if $\gcd(d,m)=1$ or $d$ is prime, then $\lvert D^\prime \rvert$ contains no irreducible divisors.
\end{cor}
\begin{proof}
Suppose $D \in \lvert D^\prime \rvert$ is an irreducible divisor.
Then $\ell(D)=\ell(D^\prime) \ge 2$. By Lemma~\ref{lem:func},
any $P \in D$ is not $\PP^1$-isolated. 
By part (II) of Theorem~\ref{thm:isolated}
we see that $P$ is imprimitive. Thus $\lvert D^\prime\rvert$
contains no primitive divisors.

Suppose now that $\gcd(d,m)=1$ or $d$ is prime. 
Then part (I)
of Theorem~\ref{thm:isolated} gives a contradiction,
therefore $\lvert D^\prime\rvert$ contains no irreducible
divisors.
\end{proof}

The following variant of Corollary~\ref{cor:isolated} has an almost
identical proof, but appealing to Theorem~\ref{thm:relativeisolated}
instead of Theorem~\ref{thm:isolated}.
\begin{cor}\label{cor:relativeisolated}
Let $\pi : C \rightarrow C^\prime$ be a morphism of curves defined over $\Q$ of degree $m \ge 2$.
Write $g$, $g^\prime$ for the genera of $C$, $C^\prime$
respectively, and suppose $g^\prime \ge 1$.
Let $d \ge 2$ be an integer satisfying \eqref{eqn:gencondrel}.
Let $D^\prime$ be an effective degree $d$ divisor 
on $C$ with $\ell(D^\prime) \ge 2$.
Then $\lvert D^\prime \rvert$ contains no 
primitive degree $d$ divisors. 
Moreover, if $\gcd(d,m)=1$ or $d$ is prime, then $\lvert D^\prime \rvert$ contains no irreducible divisors.
\end{cor}

\begin{comment}
To prove Theorems~\ref{thm:gonality} and~\ref{thm:relative}
we in fact show that if $\lvert D_i \rvert$ contains
a primitive divisor
then, under the hypotheses of either theorem, $\ell(D_i)=1$, and so $D=D_i$.
Thus the number of Galois orbits of primitive
degree $d$ points is bounded by the number
of $D_i$ in \eqref{eqn:decomp} satisfying
$\ell(D_i)=1$. 
Our next step in proving Theorems~\ref{thm:gonality}
and~\ref{thm:relative} is the following lemma.
\end{comment}

\section{Decomposition of $C^{(d)}$ into complete linear series}
\label{sec:decomp}
\begin{comment}
As in \cite[Definition 4.1]{BELOV}, we say a degree $d$ point $P\in C(\overline{\Q})$ is $\PP^1$-\textbf{isolated} 
if it doesn't lie in the image of a degree $d$ map $C\rightarrow \PP^1$ defined over $\Q$, 
and $P$ is \textbf{AV-isolated} 
if there is no positive rank subabelian variety $A\subset J$ 
such that $\iota(P)+A\subset W^{d}$.

By a result of Bourdon, Ejder, Liu, Odumodu and Viray \cite[Theorem 4.2]{BELOV}, $C$ has finitely many degree $d$ points if and only if all degree $d$ points on $C$ are isolated. Moreover, there are finitely many isolated points on $C$. As noted by Ejder \cite[Lemma 2.3]{Ejder}, if $P$ is a degree $d$ isolated point then $d\leq g$.
\end{comment}

We denote the $d$-th symmetric power of $C$ by $C^{(d)}$. Recall that
$C^{(d)}(\Q)$ can be identified with the set of 
effective degree $d$ divisors on $C$.
The purpose of this section is to prove the following proposition.
\begin{prop}\label{prop:decomp}
Let $C$ be a curve over $\Q$ of genus $g \ge 1$,
and let $J$ be the Jacobian of $C$. Let $d$ be a positive integer.
Suppose either of the following two conditions hold:
\begin{enumerate}[(a)]
	\item $J(\Q)$ is finite;
\item or $d \le g-1$,
and $A(\Q)$ is finite for every abelian subvariety $A/\Q$ of $J$ of dimension $\le d/2$.
\end{enumerate}
Then there are finitely many effective degree $d$ divisors
$D_1,D_2,\dotsc,D_n$ such that
\begin{equation}\label{eqn:decomp}
		C^{(d)}(\Q) \; = \; 
		\bigcup_{i=1}^n \lvert D_i \rvert.
\end{equation}
\end{prop}
To prove the proposition we shall need the following famous theorem of Faltings \cite{Faltings_Lang}.
%The proposition is a consequence of the following famous theorem due to Faltings \cite[Theorem 1]{FaltingsDio}.
\begin{thm}[Faltings]\label{thm:Faltings}
Let $B$ be an abelian variety defined over a number field $K$,
and let $V \subset B$ be a subvariety defined over $K$.
Then there is a finite number of abelian subvarieties
$B_1, \dotsc, B_m$ of $B$, defined over $K$, and a finite number of points $x_1,\dotsc,x_m \in V(K)$
such that the translates $x_i+B_i$ are contained in $V$,
and, moreover, such that
\begin{equation}\label{eqn:Vdecomp}
V(K)=\bigcup_{i=1}^m x_i+B_i(K).
\end{equation}
\end{thm}
We shall also need the following theorem of Debarre and Fahlaoui \cite[Corolllary 3.6]{DebarreFahlaoui}.
\begin{thm}[Debarre and Fahlaoui]
Let $C/\mathbb{C}$ be a curve of genus $g \ge 1$ with Jacobian $J$, and let $d \le g-1$.
Let $D_0$ be a fixed divisor of degree $d$ on $C$. Let $W_d(C)$
be the image of $C^{(d)}$ under the Abel--Jacobi map
\begin{equation}\label{eqn:iota}
		\iota \; : \; C^{(d)} \rightarrow J,
		\qquad D \mapsto [D-D_0].
\end{equation}
Let $A$ be an abelian subvariety of $J$ with a translate contained in $W_d(C)$.
Then $\dim(A) \le d/2$.
\end{thm}

\begin{proof}[Proof of Proposition~\ref{prop:decomp}]
If $C^{(d)}(\Q)=\emptyset$
then we take $n=0$ and there is nothing to prove. So suppose
$C^{(d)}(\Q) \ne \emptyset$ and fix
	$D_0 \in C^{(d)}(\Q)$. Let $W_d(C)$ be
the image of $C^{(d)}$ under the Abel--Jacobi map \eqref{eqn:iota},
which is defined over $\Q$.
	We claim that $W_d(C)(\Q)$ is finite. 
This is trivially true if (a) holds, so suppose (b).
In particular $d \le g-1$ and so $W_d(C)$
is birational to $C^{(d)}$ (e.g. \cite[Theorem 5.1]{MilneJ}) and so has dimension $d$.
We apply Falting's Theorem with $B=J$ and $V=W_d(C)$. 
Thus, there are $x_1,\dotsc,x_m \in W_d(C)(\Q)$
and $B_1,\dotsc,B_m$ abelian subvarieties of $J$ defined over $\Q$
such that $x_i+B_i \subset W_d(C)$ and 
\[
W_d(C)(\Q)=\bigcup_{i=1}^m x_i+B_i(\Q).
\]
By the theorem of Debarre and Fahlaoui, $\dim(B_i) \le d/2$.
However, by assumption (b), $B_i(\Q)$ is finite. Hence
$W_d(C)(\Q)$ is finite,  proving our claim.

	Let $W_d(C)(\Q)=\{R_1,\dotsc,R_n\}$ and choose $D_1,\dotsc,D_n
\in C^{(d)}(\Q)$ mapping to $R_1,\dotsc,R_n$ respectively.
	If $D \in C^{(d)}(\Q)$ then $\iota(D)=R_i$ for some $i$,
	and so $[D-D_0]=[D_i-D_0]$. Hence $[D-D_i]=0$,
	so $D \in \lvert D_i\rvert$. This completes the proof.
\end{proof}
Observe that if $C$ and $d$ satisfy the hypotheses
of either Theorem~\ref{thm:gonality} or~\ref{thm:relative}
then they satisfy the hypotheses of Proposition~\ref{prop:decomp}
and therefore $C^{(d)}(\Q)$ can be decomposed into the union of finitely many complete linear series as in \eqref{eqn:decomp}.

\section{Proofs of Theorems~\ref{thm:gonality} and~\ref{thm:relative}
and their corollaries} \label{sec:proofs}

\begin{proof}[Proof of Theorem~\ref{thm:gonality}]
Let $C$, $m$, $d$ be as in the statement of Theorem~\ref{thm:gonality}.
By Proposition~\ref{prop:decomp}, 
	%we may decompose
$C^{(d)}(\Q)$ has a finite decomposition, as in \eqref{eqn:decomp}
where $D_1,\dotsc,D_n$ are effective degree $d$ divisors.
If $\ell(D_i) \ge 2$ then, by Corollary~\ref{cor:isolated},
the complete linear series $\lvert D_i \rvert$ does not contain primitive divisors.
On the other hand, if $\ell(D_i)=1$ then $\lvert D_i \rvert =\{D_i\}$.
Hence there are only finitely many primitive degree $d$ points on $C$.

Suppose now that $\gcd(d,m)=1$ or $d$ is prime.
Again, if $\ell(D_i) \ge 2$ then, by Corollary~\ref{cor:isolated},
the complete linear series $\lvert D_i \rvert$ does not
contain irreducible divisors. The theorem follows.
\end{proof}

\begin{proof}[Proof of Corollary~\ref{cor:hyp}]
Let $C/\Q$ be hyperelliptic of genus $g$.
This has gonality $m=2$. Suppose either of 
hypotheses (a), (b) of Corollary~\ref{cor:hyp}
is satisfied. 
%Then $d$ satisfies \eqref{eqn:gencondgon}.
Then $C$, $g$, $m$, $d$ satisfy the hypotheses
of Theorem~\ref{thm:gonality}.
In particular, if $d$ is odd then $C$ has finitely
many degree $d$ points.

Suppose $d$ is even. 
By Proposition~\ref{prop:decomp},
we have that \eqref{eqn:decomp} holds
where $D_1,\dotsc,D_n$ is a finite 
%collection
set of effective degree $d$ divisors on $C$.
Let $P$ be a degree $d$ point and let
$D$ be the corresponding irreducible
divisor. Then $D \in \lvert D_i \rvert$
for some $i$. Suppose $D \ne D_i$. Then
$\ell(D) \ge 2$ and so by Lemma~\ref{lem:func}
the point $P$ is not $\PP^1$-isolated.
It follows from Theorem~\ref{thm:isolated}
that $\Q(P)$ contains a subfield
of index $d^\prime=2$. Thus, for all but
finitely many degree $d$ points $P$
we have that $\Q(P)$ contains
a subfield of index $2$.
\end{proof}

\begin{proof}[Proof of Theorem~\ref{thm:relative}]
This follows from Proposition~\ref{prop:decomp} and 
Corollary~\ref{cor:relativeisolated} by  trivial
modifications to the proof of Theorem~\ref{thm:gonality}.
\end{proof}

\begin{proof}[Proof of Corollary~\ref{cor:bielliptic}]
The proof is a straightforward modification of preceding
arguments.
%Here $m=2$. If $d=2$ or $d$ is odd then we immediately conclude
%from Theorem~\ref{thm:relative} that
%there are only finitely many degree $d$ points on $C$
%which gives (i).
%
%Suppose $d \ge 4$ is even. 
%Consider the commutative diagram \eqref{eqn:diagram2},
%where $C^\prime=E$.
%Since $\deg(\pi)=m=2$ and $\deg(h)>1$
%we conclude that $\deg(v)=1$. Hence, in the diagram
%we can take $Y=E$, $h=\pi$ and $v$ to be
%the identity map on $\PP^1$.
%	So $f=u \circ \pi$. 
\end{proof}

\subsection{A Remark on Effectivity}\label{subsec:effectivity}
Let $C$ and $d$ satisfy the hypotheses of Theorem~\ref{thm:gonality}
or Theorem~\ref{thm:relative}. Then $C^{(d)}(\Q)$
can be decomposed into a finite union of complete linear
systems as in \eqref{eqn:decomp}. Suppose that 
we are able to explicitly compute the representatives
$D_i$ in \eqref{eqn:decomp}. Then we have an effective
strategy for computing all primitive degree $d$ points.
Indeed, if $\ell(D_i) \ge 2$ then $\lvert D_i \rvert$
contains no primitive divisors by Corollary~\ref{cor:isolated}.
We are left to consider
$\lvert D_i \rvert$ for $\ell(D_i)=1$. However, if $\ell(D_i)=1$,
then $\lvert D_i \rvert=\{D_i\}$ and we simply need to test
$D_i$ to determine if it is the Galois orbit
of a primitive degree $d$ point. Moreover, if $\gcd(d,m)=1$ or $d$ is prime then we can compute all degree $d$ points by a slight modification
of the strategy: if $\ell(D_i)=1$ then simply
test $D_i$ for irreducibility.

We remark that the decomposition~\eqref{eqn:decomp}
can often be computed using symmetric power Chabauty
(e.g. \cite{ChabautySym} or \cite{BGG22}) provided $r+d \le g$ where
$r$ is the rank of the Mordell--Weil group $J(\Q)$.

\subsection{A Remark on the Assumptions of Corollary~\ref{cor:hyp}}
As noted previously, assumption (b) of Corollary~\ref{cor:hyp} is satisfied for $d \le g-1$
whenever the Jacobian $J$ is simple.
We remark that almost all hyperelliptic curves have simple Jacobians,
in a sense that we make precise shortly. 
For this we will need the following theorem of Zarhin~\cite[Theorem 1.1]{Zarhin2009}.
\begin{thm}[Zarhin]\label{thm:Zarhin}
Let $k$ be a field, $char(k)\neq 2$. 
Let $C: Y^2=f(X)$, where  $f\in k[X]$ is a separable 
polynomial of degree $n \ge 5$, and let $J$ denote the Jacobian of $C$.  Suppose
$char(k)\neq 3$ or $n\geq 7$, and 
that $f$ has Galois group either $S_{n}$ or $A_{n}$.
Then $\End(J)\cong\Z$, and in particular $J$ is absolutely simple.
\end{thm}
We fix a genus $g \ge 2$. Let $n=2g+1$ or $2g+2$.
The set of all polynomials of degree $\le n$ can be naturally identified with
$\Aff^{n+1}(\Q)$: here $\aaa=(a_0,a_1,\dotsc,a_{n}) \in \Aff^{n+1}(\Q)$
corresponds to the polynomial
\[
	f_\aaa (X) \; = \; a_n X^n+a_{n-1} X^{n-1}+\cdots+a_0.
\]
Hilbert's irreducibility theorem asserts the existence of a thin set $S \subset \Aff^{n+1}(\Q)$
such that for $\aaa \in \Aff^{n+1}(\Q) \setminus S$, the polynomial $f_\aaa$ is irreducible
of degree $n$ and has Galois group $S_n$. 
See \cite[Chapter 9]{SerreMW} for a statement of Hilbert's irreducibility theorem
as well as the definition of thin sets.

Therefore, for $\aaa \in \Aff^{n+1}(\Q) \setminus S$,
the genus $g$ hyperelliptic curve $Y^2=f_\aaa(X)$ has
 a simple Jacobian by Zarhin's theorem.

We point out that there is no shortage of
hyperelliptic curves satisfying the finite
Mordell--Weil group condition of Corollary~\ref{cor:hyp}.  This immediately
follows from the preceding remarks together with the following Theorem of Yu
\cite[Theorem 3]{Yu16}, applied with $r=0$.

\begin{thm}[Yu]\label{thm:Yu}
Let $K$ be a number field with at least one real embedding.
Let $f\in K[X]$ be a separable degree $n$ polynomial such that
$n\equiv 3 \pmod{4}$ and $\Gal(f)\cong S_{n}$ or $A_{n}$.
Let $C:\; Y^2=f(X)$ be a hyperelliptic curve over $K$ with Jacobian $J$.
Then for every $r\geq 0$, there are infinitely many quadratic twists $J_{b}$ of $J$ such that
$\dim_{\F_{2}}(\Sel_{2}(J_{b}/K))=r$.
\end{thm}

\section{Proof of Theorem ~\ref{thm:infprimitive}}\label{sec:infprimitive}
The proof of Theorem~\ref{thm:infprimitive} relies on the following proposition, which is in fact
a consequence of Hilbert's Irreducibility Theorem.
\begin{prop}\label{prop:morph}
Let $d\geq 2$ be an integer.
Let $f : C \rightarrow \PP^1$ be a degree $d$ morphism of curves defined over $\Q$, 
and let $\alpha \in \PP^1(\Q)$. Suppose that $\alpha$ is not a branch value for $f$,
and that the fibre $f^{-1}(\alpha)$ consists of a single 
Galois orbit of points. Let $P \in f^{-1}(\alpha)$ and suppose the extension $\Q(P)$ is primitive.
Then there is a thin set $S \subset \PP^1(\Q)$ such that for $\beta \in \PP^1(\Q) \setminus S$,
the fibre $f^{-1}(\beta)$ consists of a single Galois orbit of points and,
for any $Q \in \pi^{-1}(\beta)$, the extension
$\Q(Q)$ is primitive of degree $d$.
\end{prop}
%For the definition of a thin set see for example \cite[Section 9.1]{SerreMW}. 
Before proving Proposition~\ref{prop:morph}
we show that it implies Theorem~\ref{thm:infprimitive}.
\begin{proof}[Proof of Theorem~\ref{thm:infprimitive}]
Let $C/\Q$ be a curve of genus $g$. 
Suppose $P \in C(\overline{\Q})$ is primitive of degree $d \ge g+1$,
and let $D$ be the corresponding irreducible divisor.

By Riemann--Roch,
\[
	\ell(D) \; \ge \; d-g+1 \; \ge \; 2.
\]
It follows from Lemma~\ref{lem:func}
that there is a degree $d$ morphism
$f : C \rightarrow \PP^1$ defined over $\Q$
such that $f^*(\infty)=D$.
We apply Proposition~\ref{prop:morph} with $\alpha=\infty \in \PP^1(\Q)$.
The theorem follows.
\end{proof}

The proof of Proposition~\ref{prop:morph} makes use
of the relationship between primitive extensions and
primitive Galois groups. Whilst this relationship
is known, we are unable to find a convenient
reference, and we therefore give the details.
Let $X$ be a non-empty set, and let $G$ be
a group acting transitively on $X$. 
The \textbf{trivial} partitions of $X$ 
are $\{X\}$ and $\{\{x\} : x \in X\}$.
A partition $\cP$ of $X$ is said to be \textbf{$G$-stable}
if $\sigma(Z) \in \cP$ for all
$\sigma \in G$ and $Z \in \cP$. 
Observe, as the action of $G$ on $X$ is transitive,  
that $G$ also acts transitively on any $G$-stable partition $\cP$,
and that any two elements of $\cP$ therefore have the same cardinality.

We say that $G$ \textbf{acts imprimitively}
on $X$ if $X$ admits a $G$-stable non-trivial partition.
If $X$ does not have a $G$-stable non-trivial partition
then we say that $G$ \textbf{acts primitively on $X$}.
The following lemma is an immediate consequence of this
definition.
\begin{lem}\label{lem:bigger}
Let $G$ be a group acting transitively on a set $X$.
Let $G^\prime$ be a subgroup of $G$ and suppose that $G^\prime$
acts primitively on $X$. Then $G$ acts primitively on $X$.
\end{lem}
The following result is well-known, and in particular 
implies that $S_d$ and $A_d$ act primitively on $\{1,\dotsc,d\}$,
for $d \ge 1$ and $d \ge 3$ respectively.
\begin{lem}\label{lem:Sdprim}
Let $G$ be a group acting $2$-transitively on a set $X$.
Then the action is primitive.
\end{lem}
\begin{proof}
Let $\cP$ be a $G$-stable partition of $X$ and suppose $Y \in \cP$ has at least two
elements. We would like to show that $Y=X$.
Let $a$, $b \in Y$ be distinct, and let $c \in X$ be distinct from $a$, $b$. Then there is some $\sigma \in G$
such that $\sigma(a)=a$ and $\sigma(b)=c$. Thus $a \in Y \cap \sigma(Y)$ which forces $\sigma(Y)=Y$, and therefore $c \in Y$. Hence $Y=X$. 
\end{proof}

\begin{lem}\label{lem:transvection}
Let $G$ be a group acting transitively on a set $X$.
The action is imprimitive if and only if there is 
a proper subset $Y$ of $X$,
with at least two elements, such that
\begin{equation}\label{eqn:transvection}
	\forall \sigma \in G, \qquad \text{if $\sigma(Y) \cap Y \ne 
	\emptyset$ then $\sigma(Y)=Y$}.
\end{equation}
\end{lem}
\begin{proof}
Given a $G$-stable non-trivial partition $\cP$ we can take
$Y$ to be any element of $\cP$. As $\cP$ is a partition, $Y$ clearly satisfies
\eqref{eqn:transvection}, and as $\cP$ is non-trivial, $Y$ is a proper subset of $X$
with at least two elements.

Conversely, suppose $Y$ is a proper subset of $X$ containing
at least two elements and satisfying \eqref{eqn:transvection}.
We easily check that $\cP=\{\tau(Y) : \tau \in G\}$ is a $G$-stable
non-trivial partition.
\end{proof}

\begin{lem}\label{lem:primstab}
Let $G$ be a finite group acting transitively on a non-empty finite set $X$.
Let $x \in X$, and write $\Stab(x)$ for the stabilizer
of $x$ in $G$. 
The action of $G$ on $X$ is imprimitive
	if and only if $\Stab(x)$ is a non-maximal
	proper subgroup of $G$.
\end{lem}
\begin{proof}
Let $x \in X$ and assume the existence of a subgroup $\Stab(x) \subsetneq H \subsetneq G$.
Let $Y=\{ \tau(x) : \tau \in H\}$. Then, $\# Y=[H:\Stab(x)]$ and so
$2 \le \# Y < [G: \Stab(x)]=\# X$. 
Moreover, suppose $\sigma \in G$ and $\sigma(Y) \cap Y \ne 
	\emptyset$. 
Let $z \in \sigma(Y) \cap Y$.
Then there are $\tau_1$, $\tau_2 \in H$ such that $\sigma\tau_2(x)=z=\tau_1(x)$. Thus $\tau_1^{-1} \sigma \tau_2 \in \Stab(x) \subseteq H$.
	Hence $\sigma \in H$, and so $\sigma(Y)=Y$. Therefore \eqref{eqn:transvection} is satisfied and so the action is imprimitive.

Conversely, suppose the existence of a proper subset $Y$ of $X$ with at least two elements satisfying \eqref{eqn:transvection}. 
As the action is transitive, we may in fact suppose that $x \in Y$.
	Let
	$H=\{\sigma \in G \; : \: \sigma(Y)=Y\}$. As $G$ is transitive, $H$ is a proper subgroup of $G$. Moreover, $\Stab(x)$ is contained in $H$.
Let $x^\prime \in Y$, $x^\prime \ne x$. Then there is some
	$\sigma \in G$ such that $\sigma(x)=x^\prime$. Thus $\sigma(Y)=Y$,
	and so $\sigma \in H$ but $\sigma \notin \Stab(x)$. It follows
	that $\Stab(x)$ is a proper in $H$, and so is non-maximal
	as a subgroup of $G$.
\end{proof}

\begin{lem}\label{lem:primGal}
Let $K=\Q(\theta)$ be a number field and let $\tilde{K}$ be its Galois
closure over $\Q$. 
Let $G=\Gal(\tilde{K}/\Q)$.
	Let $d=[K:\Q]$
	and let $\theta_1,\dotsc,\theta_d \in \tilde{K}$ be the Galois
conjugates of $\theta$. Then $G$ acts primitively
on $\{\theta_1,\dotsc,\theta_d\}$ if and only if 
the extension $K/\Q$ is primitive.
\end{lem}
\begin{proof}
Let $X=\{\theta_1,\dotsc,\theta_d\}$. Then $G$
acts transitively on $X$. We let $x=\theta \in X$ 
and note that the stabilizer $\Stab(\theta)$ is in fact
$\Gal(\tilde{K}/K)$. By the Galois correspondence, $K$ is imprimitive
	if and only if the subgroup $\Gal(\tilde{K}/K)$ is proper
	and non-maximal in $G$, which by Lemma~\ref{lem:primstab}
	if equivalent to the action of $G$ being imprimitive.
\end{proof}

The following result was assumed in the introduction.
\begin{lem}\label{lem:pointSdprim}
Let $C/\Q$ be a curve, let $d \ge 3$ and let $P$
be a degree $d$ point on $C$ with Galois group $S_d$
or $A_d$. Then $P$ is primitive. 
\end{lem}
\begin{proof}
The lemma follows from Lemma~\ref{lem:primGal},
and for this we need the fact that $S_d$
and $A_d$ act primitively on $\{1,\dotsc,d\}$.
This is trivially true if $d=3$, and for $d \ge 4$
follows from Lemma~\ref{lem:Sdprim}.
\end{proof}

\begin{proof}[Proof of Proposition~\ref{prop:morph}]
By composing $f$ with a suitable automorphism of $\PP^1$
we may suppose that $\alpha=0 \in \PP^1(\Q)$.
Write $\KK=\Q(C)$ for the function field of $C$.
We may regard $f$ as a non-constant element of $\KK$,
and with this interpretation $\Q(f)=\Q(\PP^1)$ is a subfield of $\KK$
of index $d$. 
	
By hypothesis, $\alpha=0$ is not a branch value for $f$,
and $f^{-1}(0)$ consists of a single Galois orbit containing $P$.
Write $D=f^*(0)$ which we think of as a degree $d$ place of $\KK$,
unramified in the extension $\KK/\Q(\PP^1)$.
Write
\[
	\OO_{D}=\{ h \in \KK \; : \; \ord_D(h) \ge 0\}, 
	\qquad
	\fm_{D}=\{ h \in \KK \; : \; \ord_D(h) \ge 1\},
\]
for the valuation ring of $D$ and its maximal ideal. Then the residue field 
$\OO_{D}/\fm_{D}$ can be identified with $K=\Q(P)$ where
the identification is given by $g +\fm_{D} \mapsto g(P)$.
	Now fix $\theta \in K$ such that $K=\Q(\theta)$.
Then there is some $g \in \OO_{D}$ such that
$g(P)=\theta$. As $g \in \KK$ and $\KK$ has degree $d$ over $\Q(f)$,
there is a polynomial $F(U,V) \in \Q[U,V]$,
\begin{equation}\label{eqn:FUV}
	F(U,V)=\sum_{i=1}^n a_i(V) U^i, \qquad a_i(V) \in \Q[V]
\end{equation}
of degree $n \mid d$, such that $\gcd(a_0(V),\dotsc,a_n(V))=1$,
and $F(g,f)=0$. 
Now, $F(\theta,0)=F(g(P),f(P))=0$, and so $\theta$ is a root of the polynomial $F(U,0) \in \Q[U]$;
this polynomial is non-zero as $\gcd(a_0(V),\dotsc,a_n(V))=1$.
As $\theta$ has degree $d$, it follows that $n=d$, and that $F(U,V)$ is irreducible over $\Q(V)$.
In particular $F(U,V)=0$ is a (possibly singular) plane model for $C$. As $C$
is absolutely irreducible, $F(U,V)$ is irreducible over $\overline{\Q}$. 
Write $\tilde{\KK}$ for the Galois closure of $\KK/\Q(\PP^1)$,
and let $g_1,\dotsc,g_d$ be the roots of $F(U,f)=0$ in $\tilde{\KK}$;
then $\tilde{\KK}=\Q(\PP^1)(g_1,\dotsc,g_d)$.

Let $\tilde{C}/\Q$ be the algebraic curve associated to $\tilde{\KK}$.
Let $\tilde{D}$ be an extension of the place $D$ of $\KK$ to $\tilde{\KK}$.
As $D$ is unramified in $\KK/\Q(\PP^1)$, the place $\tilde{D}$ is unramified
in the Galois closure $\tilde{\KK}/\Q(\PP^1)$ (see for example \cite[Corollary III.8.4]{Stichtenoth}).
We claim that $g_1,\dotsc,g_d \in \OO_{\tilde{D}}$. To see this, note that
	$G=\Gal(\overline{\KK}/\Q(\PP^1))$ acts transitively on the $g_i$.
	Thus we would like to show $\ord_{\tilde{D}}(\tau(g)) \ge 0$
	for all $\tau \in G$. However, 
	\[
		\ord_{\tilde{D}}(\tau(g)) \; = \; \ord_{\tau^{-1}(\tilde{D})}(g) \; = \; \ord_{D}(g) \; \ge \; 0,
	\]
where the second equality follows as $g \in \KK$ and $\tau^{-1}(\tilde{D})$ is an unramified place of $\tilde{\KK}$ above $D$.
This proves the claim.

The place $\tilde{D}$
corresponds to a Galois orbit of points on $\tilde{C}$
and we let $\tilde{P}$ be a representative point chosen above $P$.
	Write $G_{\tilde{D}}$ for the decomposition group corresponding to $\tilde{D}$
in $G=\Gal(\tilde{\KK}/\Q(\PP^1))$; by definition this is
\[
	G_{\tilde{D}}=\{\sigma \in G \; : \; \sigma(\fm_{\tilde{D}})=\fm_{\tilde{D}} \}.
\]
As $\tilde{D}$ is unramified in 
the Galois closure $\tilde{\KK}/\Q(\PP^1)$, the theory of decomposition groups
allows us to identify the 
decomposition group $G_{\tilde{D}}$ with $\Gal(\tilde{K}/\Q)$, where $\tilde{K}$
is the Galois closure of $K/\Q$. Here we shall in fact 
need the details of the precise identification. Let $\sigma \in G_{\tilde{D}}$.
We associate $\sigma$ to $\sigma^\prime \in \Gal(\tilde{K}/\Q)$ as follows.
We let $\gamma \in \tilde{K}\cong \OO_{\tilde{D}}/\fm_{\tilde{D}}$. Then
$\gamma=h(\tilde{P})$ for some $h \in \OO_{\tilde{D}}$,
and we define $\sigma^\prime(\gamma)=\sigma(h)(\tilde{P})$. It immediately
follows from the definition of $G_{\tilde{D}}$ that $\sigma^\prime$ is well-defined. The map $G_{\tilde{D}} \rightarrow \Gal(\tilde{K}/\Q)$
given by $\sigma \mapsto \sigma^\prime$ is in fact an isomorphism \cite[Theorem III.8.2]{Stichtenoth},
and from now on we shall 
identify $\Gal(\tilde{K}/\Q)$ with $G_{\tilde{D}}$ via this identification.
	
Let $\theta_1,\dotsc,\theta_d$ be the conjugates of $\theta$ in $\tilde{K}$, where we take $\theta_1=\theta$.
Now $\theta=g(P)=g(\tilde{P})$. Let $1 \le i \le d$. Then there is some $\sigma_i \in G_{\tilde{D}}$ such that
$\sigma_i(\theta)=\theta_i$ and therefore $\sigma_i(g)(\tilde{P})=\theta_i$. 
Thus there is a conjugate
of $g$ sending $\tilde{P}$ to $\theta_i$. 
As the $\theta_i$ are pairwise distinct, we may after reordering the $g_i$ suppose that
$g_i(\tilde{P})=\theta_i$. Now $\Gal(\tilde{K}/\Q)=G_{\tilde{D}}$ is a subgroup of $G=\Gal(\tilde{\KK}/\Q(\PP^1))$
and both act on $\{g_1,\dotsc,g_d\}$. As $K/\Q$ is primitive, Lemma~\ref{lem:primGal} tells us that 
$G_{\tilde{D}}$ acts primitively on the 
$\theta_i$ and therefore on the $g_i$. Thus by Lemma~\ref{lem:bigger}, the action
of $G$ on $\{g_1,\dotsc,g_d\}$ is primitive.

Finally we apply Hilbert's Irreducibility Theorem \cite[Chapter 9]{SerreMW} to $f: C \rightarrow \PP^1$. 
For $\beta \in \PP^1(\Q)$ we shall make a standard abuse of language
and speak of the decomposition group $G_\beta$ by which we mean the 
decomposition group $G_{\tilde{Q}}$ for any point $\tilde{Q} \in \tilde{C}$ above $\beta$.
As usual, $G_\beta$ is only defined up to conjugation in $G$.
Now, Hilbert's Irreducibility Theorem applied to $f: C \rightarrow \PP^1$
asserts the existence of a thin set $S \subset \PP^1(\Q)$,
which includes all branch values, such that
$G_\beta=G$ for $\beta \in \PP^1(\Q) \setminus S$. We enlarge $S$ by adjoining
finitely many values in $\PP^1(\Q)$: the roots of $a_d(V)$
(which is the leading coefficient of $F$ regarded as an element
of $\Q[V][U]$); the $V$-coordinate of any singular point
of the plane model $F(U,V)=0$; and the point $\infty \in \PP^1(\Q)$.
As we have added finitely many points to the set $S$ it remains thin.
Let $\beta \in \PP^1(\Q) \setminus S$. Let $\phi_1,\dotsc,\phi_d$
be the roots of $F(U,\beta)$. Let $Q=(\phi_1,\beta)$; this is a smooth point on
the plane model, and so may be regarded as a point on $C$. Let $\tilde{Q}$
be a point of $\tilde{C}$ above $Q$. Then as before, we can identify the 
action of $G_\beta$ on $\{\phi_1,\dotsc,\phi_d\}$ with the action of $G$
on $\{g_1,\dotsc,g_d\}$. As $G$ is acting transitively and primitively on 
$\{g_1,\dotsc,g_d\}$,
we have that $G_\beta=\Gal(\Q(f^{-1}(\beta))/\Q)$ is acting
transitively and primitively on $\{\phi_1,\dotsc,\phi_d\}$. Hence, the Galois action
on the fibre $f^{-1}(\beta)=\{(\phi_1,\beta),\dotsc,(\phi_d,\beta)\}$ is
primitive and, by Lemma~\ref{lem:primGal}, the field $\Q(Q)=\Q(\phi_1)$ is primitive of 
degree $d$.
The proposition follows.
\end{proof}

Let $C/\Q$ be a curve and let $d \ge g+1$ where $g$ is the genus of $C$.
Theorem~\ref{thm:infprimitive} asserts the existence 
of infinitely many primitive degree $d$ points on $C$
provided there is at least one. However, the existence
of a primitive degree $d$ point is not guaranteed,
as illustrated by the following lemma.
\begin{lem}
Let $g \ge 2$ be even. Let $C$ be a degree $2g+1$
genus $g$ curve defined over $\Q$
\[
        C \; : \; Y^2 \; = \; a_{2g+1} X^{2g+1}+a_{2g} X^{2g}+\cdots+a_0.
\]
Suppose $J(\Q)=0$ where $J$ is the Jacobian of $C$.
Then $C$ has no points of degree $g+1$ points.
\end{lem}
\begin{proof}
Write $\infty$ for the single point at infinity
on the given model. Write $D_0=(g+1) \infty$.
Note that $X$ has a double pole at $\infty$.
Thus $1,X,\dots,X^{g/2} \in L(D_0)$. 
We claim that $1,X,\dotsc,X^{g/2}$ is a basis for $L(D_0)$.
We first explain how our claim implies the lemma.
Let $D$ be an effective degree $g+1$ divisor.
Since $J(\Q)=0$, we have $D-D_0=\divv(f)$
for some $f \in L(D_0)$. Thus,
$f=\alpha_0+\alpha_1 X+\cdots+\alpha_{g/2} X^{g/2}$, for some $\alpha_0,\dotsc,\alpha_{g/2} \in \Q$.
In particular $f \in L(g \infty)$.
        Thus
        \[
                D-\infty \; =\; D_0+\divv(f)-\infty \; = \; g \infty +\divv(f)
        \]
is effective. Hence $D$ is reducible. It follows that $C$ has no degree $g+1$ points.

It remains to prove our claim.
Since $1,X,\dots,X^{g/2} \in L(D_0)$, our claim is equivalent to showing that
$\ell(D_0) \le g/2+1$. 
If $g=2$, the Riemann--Roch theorem
immediately implies that $\ell(D_0)=2=g/2+1$ as required.
We may therefore suppose $g >2$, and as $g$ is even, $g \ge 4$.
It follows from the Riemann--Roch
theorem \eqref{eqn:RR} that $\ell(K_C-D_0)=\ell(D_0)-2 \ge g/2-1>0$.
In particular, $D_0$ is a special divisor.
By Clifford's theorem (Theorem~\ref{thm:Clifford})
we have $\ell(D_0) \le g/2+3/2$. However, since $g$ is even and $\ell(D_0)$
is an integer, we have $\ell(D_0) \le g/2+1$,
completing the proof.
\end{proof}

In a positive direction, we can use Theorem~\ref{thm:infprimitive}
to construct curves with infinitely many primitive points.
\begin{lem}
Let $g \ge 2$. Let $d=g+1$. Then there is a hyperelliptic curve $C/\Q$
of genus $g$ with infinitely many primitive degree $d$ points.
\end{lem}
\begin{proof}
Let $K=\Q(\theta)$ be any primitive number field of degree $d$.
Let $\theta_1,\dotsc,\theta_d$ be the conjugates of $\theta$
in a fixed Galois closure $\tilde{K}$ of $K$.
Choose a rational number $\alpha$ such that
$2\alpha \ne \theta_i+\theta_j$ for any pair $1 \le i,~j \le d$. 
Let $\phi=\theta-\alpha$. The conjugates of $\phi$
are $\phi_i=\theta_i-\alpha$ with $1 \le i \le d$,
and satisfy $\phi_i \ne \pm \phi_j$ for  any pair $i$, $j$.
Note that $\Q(\phi^2) \subseteq K$.
As $K$ is primitive, either $\Q(\phi^2)=\Q$
or $\Q(\phi^2)=K$.  However, if $\Q(\phi^2)=\Q$,
then $K=\Q(\theta)=\Q(\phi)$ has degree at most
$2$, contradicting the fact that $K$ has
degree $d=g+1 \ge 3$. Thus $K=\Q(\phi^2)$.
Let $f \in \Q[X]$ be the minimal polynomial
of $\phi^2$, which must be
irreducible of degree $d$. 
Let $h=f(X^2)$. 
The roots of $h$ are $\pm \phi_1,\dotsc,\pm \phi_d$
which are pairwise distinct and hence $h$ is separable
of degree $2d=2g+2$. Let $C$ be the 
genus $g$ hyperelliptic
curve 
\[
	C \; : \; Y^2 \; = \; h(X).
\]
Note that this has the primitive
degree $d$ point $(\phi,0)$. Hence by
Theorem~\ref{thm:infprimitive} there are
infinitely many primitive degree $d$ points.
\end{proof}

\section{Low degree primitive points on some $X_1(N)$}\label{sec:modular1}
%\subsection{Low degree primitive points on $X_{1}(N)$}
Mazur \cite{Mazur_Eisenstein} showed that all rational points on $X_1(p)$
are cuspidal for prime $p  \ge 11$.
Merel's uniform boundedness theorem \cite{Merel} asserts that for prime $p$,
and for $d$ satisfying $(3^{d/2}+1)^2 \le p$, the only
degree $d$ points on $X_1(p)$ are cuspidal.
We now know, for each $1 \le d \le 8$,  the exact set
of primes $p$ such that $X_1(p)$ has degree $d$ non-cuspidal points,
thanks to the efforts of Kamienny \cite{Kamienny}, Parent \cite{Parent}, \cite{Parent17},
Derickx, Kamienny, Stein and Stoll \cite{DKSS}, and Khawaja \cite{Khawaja}.
Less is known about the low degree points on $X_1(N)$ for composite $N$,
though several authors consider the somewhat easier problem of determining
the values of $N$ such $X_1(N)$ has infinitely many
degree $d$ points for given small $d$ (see for example \cite{BELOV} and \cite{DerickxSutherland}
for two different approaches to studying this problem).
Example~\ref{ex:X1(45)} illustrates how our results can be applied
to modular curves $X_1(N)$ provided the analytic rank of $J_1(N)$
is $0$ and we have information about the quotients or gonality of $X_1(N)$.
The \texttt{LMFDB}~\cite{lmfdb} contains a database of modular curves $X_{1}(N)$ for $1\leq N\leq 70$. For $61$ of 
these curves the Jacobian $J=J_{1}(N)$ has analytic rank $0$.
It follows from a theorem of Kato \cite[Corollary 14.3]{Kato} 
that the Mordell--Weil group $J(\Q)$ is finite. 
We are able to apply Theorem~\ref{thm:relative} to around half of these curves 
in order to deduce the finiteness of primitive points of certain low degrees.
We note that it is common for $X_{1}(N)$ to cover multiple curves, and in these instances we apply 
Theorem~\ref{thm:relative} to the covered curve $C^{\prime}$ that
gives the most generous range for $d$ in inequality \eqref{eqn:gencondrel}.
%Theorem~\ref{thm:relative} combined with the information in the \texttt{LMFDB}
%did not give any information about low degree points on $X_{1}(44)$. However,
%according to the \texttt{LMFDB}, the $\Q$-gonality of $X_{1}(44)$ is bounded above 
%by $16$, and this allowed to apply Theorem~\ref{thm:gonality} to conclude that $X_1(44)$
%has finitely many quadratic and cubic points.
We record the results in Table~\ref{tab:X1n}.

\begin{table}

\begin{tabular}{|c |c c c c c c|} 
 \hline
 $N$ & $g$ 
 & \makecell{$C^\prime$
\\(\texttt{LMFDB} label)}
 & $g^\prime$ & $m$ & 
 \makecell{$X_{1}(N)$ has 
 finitely many 
 \\ degree $d$ points}
& 
 \makecell{$X_{1}(N)$ has 
 finitely many 
 \\primitive 
 degree $d$ points} \\ 
 \hline\hline
 19 & 7 & 19.120.1-19.a & 1 & 3 & $d = 2$ & - \\
 \hline
 22 & 6 & $X_{1}(11)$ & 1 & 3 & $d = 2$ & - \\
 \hline
 24 & 5 & 24.192.1-24.dg.2.1 & 1 & 2 & $2\leq d\leq 3$& - \\
 \hline
 26 & 10 & $X_{1}(13)$ & 2 & 3 & $d = 2$ & - \\
 \hline
 27 & 13 & 27.216.1-27.a.1.1 &1 & 3 &$2\leq d\leq 5$ & - \\
 \hline
 28 & 10 & 28.288.4-28.d.1.1 & 4 & 2 & $d = 2$ & - \\
 \hline
 30 & 9 & $X_{1}(15)$ & 1 & 3 & $2\leq d\leq 3$ & - \\
 \hline
 31 & 26 & 31.320.6-31.c.1.2 & 6 & 3 & $2\leq d\leq 4$ & - \\
 \hline
 32 & 17 & 32.384.5-32.bu.1.1 & 5 & 2 & \makecell{$2\leq d\leq 7$\\
 $d\neq 4, 6$} & $d = 4, 6$ \\
 \hline
 34 & 21 & $X_{1}(17)$ & 5 & 3 &  $2\leq d\leq 3$ & - \\
 \hline 
 36 & 17 & 36.288.3-36.c.1.1 & 3 & 3 & $2\leq d\leq 4$ & - \\ %changed
 \hline
 38 & 28 & $X_{1}(19)$ & 7 & 3 & $2\leq d\leq 4$ & - \\
 \hline
 39 & 33 & 39.448.9-39.a.3.1 & 9 & 3 & $2\leq d\leq 3$ & - \\
 \hline
 40 & 25 & 40.576.9-40.bh.1.1 & 9 & 2 & \makecell{$2\leq d\leq 7$\\
 $d\neq 4, 6$} & $d = 4, 6$ \\
 \hline
 42 & 25 & $X_{1}(21)$ & 5 & 3 & $2\leq d\leq 5$& - \\
 \hline
 %using THEOREM 5
% 44 & 36 & 4.288.8-44.c.4.2 & 8 & 5 & $2\leq d\leq 3$ & - \\
44 & 36 & 44.720.16-44.e.1.1 & 16 & 2 & $2 \le d \le 3$ & $d=4$\\ %changed 
 \hline
 45 & 41 & 45.576.9-45.a.4.1 & 9 & 3 & \makecell{$2\leq d\leq 7$\\
 $d\neq 6$} & $d=6$ \\
 \hline
 46 & 45 & $X_{1}(23)$ & 12 & 3 & $2\leq d\leq 5$ & - \\
 \hline
 48 & 37 & 48.768.13-48.nt.1.1 & 13 & 2 & \makecell{$2\leq d\leq 11$\\
 $d\neq 4, 6, 8, 10$} & $d = 4, 6, 8, 10$ \\
 \hline
 49 & 69 & 49.336.3-49.b.1.2 & 3 & 7 & $2\leq d\leq 8$ & -\\
 \hline
 50 & 48 & 50.360.4-50.a.2.2 & 4 & 5 &$2\leq d\leq 7$ & - \\
 \hline
 52 & 55 & 52.1008.25-52.p.1.1& 25 & 2 & \makecell{$2\leq d\leq 5$\\
 $d\neq 4$} & $d=4$ \\
 \hline
 54 & 52 & 54.648.10-54.a.1.1 & 10 & 3 & \makecell{$2\leq d\leq 11$\\
 $d\neq 6, 9$} & $d= 6, 9$ \\
 \hline
 56 & 61 & 56.1152.25-56.bq.1.1 & 25 & 2 & \makecell{$2\leq d\leq 11$\\
 $d\neq 4, 6, 8, 10$} & $d = 4, 6, 8, 10$\\
 \hline
 60 & 57 & 60.1152.25-60.eb.2.1 & 25 & 2 & \makecell{$2\leq d\leq 7$\\
 $d\neq 4, 6$} & $d = 4, 6$\\
 \hline
 62 & 91 & $X_{1}(31)$ & 26 & 3 & \makecell{$2\leq d\leq 7$\\
 $d\neq 6$} & $d = 6$ \\
 \hline
 64 & 93 & 64.1536.37-64.ef.1.1 & 37 & 2 & 
 \makecell{$d=2$\\
 $3\leq d\leq 19$, odd $d$ }
 & \makecell{$4\leq d\leq 18$\\
 even $d$}\\
 \hline
 66 & 81 & $X_{1}(33)$ & 21 & 3 & \makecell{$2\leq d\leq 9$\\
 $d\neq 6, 9$} & $d=6, 9$\\
 \hline
 68 & 105 & 68.1728.49-68.ba.1.1 & 49 & 2 &\makecell{$2\leq d\leq 7$\\
 $d\neq 4, 6$} & $d=4, 6$\\
 \hline
 70 & 97 & $X_{1}(35)$ & 25 & 3 & \makecell{$2\leq d\leq 11$\\
 $d\neq 6, 9$} & $d = 6, 9$\\ % removed d=12
 \hline
\end{tabular}
\caption{The table summarizes our conclusions upon applying 
Theorem~\ref{thm:relative} to $C=X_1(N)$ for the values of $N$ in the first column.
Here $g$ denotes the genus of $X_{1}(N)$; the integer $m$ denotes the degree of the morphism 
$X_{1}(N)\rightarrow C^\prime$; $g^\prime$ denotes the genus of $C^\prime$.
The sixth column gives the values of $d$ furnished by the theorem
for which there are only finitely many
points of degree $d$. The final column gives the values of $d$
(not appearing in the previous column) for which the theorem
asserts that there are only finitely primitive degree $d$ points.}
\label{tab:X1n}
\end{table}

In 
\cite[Theorem 3.1]{Derickx21} the authors give
a complete list
of $N$ for which $J_1(N)$ has analytic rank $0$.
There are in total $83$ such values of $N$,
the largest of which is $N=180$.

\section{Low degree primitive points on some $X_0(N)$}\label{sec:modular2}
The computational study of quadratic points on modular curves is an active area of research
(see e.g. \cite{AKMJNOV2023}, \cite{bruin_najman_2015}, \cite{FLS}, \cite{NV23}, \cite{OzmanSiksek}, to name but
a few works).  Comparatively less is known about points defined over number
fields of higher degree.
Still, there is reason to be hopeful. 
Establishing the modularity of all elliptic curves over totally real cubic fields \cite{DNS20}, and totally real quartic fields not containing $\sqrt{5}$ \cite{Box22} required the study of cubic, and quartic points on certain modular curves.
Banwait and Derickx
\cite{BD22} have determined all cubic points on $X_{0}(N)$ for 
$N\in\{41, 47, 59, 71\}$.
Box, Gajovi\'{c}, and Goodman \cite{BGG22} have determined all cubic points on $X_{0}(N)$ for $N\in\{53,57,61,65,67,73\}$, and all quartic points on $X_{0}(65)$.
A famous theorem of Ogg \cite{Ogg} asserts that there are $19$ values of $N$ for which which $X_0(N)$ is hyperelliptic.
Of these, the only one for which $J_0(N)(\Q)$ is infinite is $N=37$. The remaining $18$ values are
\begin{itemize}
\item genus $2$: $N=22$, $23$, $26$, $28$, $29$, $31$, $50$;
\item genus $3$: $N= 30$, $33$, $35$, $39$, $40$, $41$, $48$; 
\item genus $4$: $N=47$;
\item genus $5$: $N=46$, $59$;
\item genus $6$: $N=71$.
\end{itemize}
For these $N$, the quadratic points on $X_{0}(N)$ have been
determined by Bruin and Najman \cite{bruin_najman_2015}.  
It is easy to apply Corollary~\ref{cor:hyp} to these curves
and derive conclusions about algebraic points of degree $3 \le d \le g$, where $g$ is the genus of $X_0(N)$. 
For example, consider $C=X_0(71)$ with genus $g=6$.
By Corollary~\ref{cor:hyp} we know that there are only finitely many points on $X_0(71)$ of degrees $3$ and $5$,
and finitely many primitive points of degrees $4$, $6$. We point out that we can in fact go further
and compute these finite sets of points, as sketched in Remark~\ref{subsec:effectivity}.
We illustrate this by giving some details of the computation of 
primitive degree $6$ points on $X_0(71)$, making use
of information found in \cite{bruin_najman_2015} concerning
the model and the Mordell--Weil group. A model for $X_0(71)$ is given by
\begin{multline*}
	X_0(71) \; : \; Y^2 \; = \; X^{14} + 4 X^{13} - 2 X^{12} - 38 X^{11} - 77 X^{10} 
    - 26 X^9 + 111 X^8\\ + 148 X^7 + X^6 - 122 X^5 - 70 X^4 + 30 X^3 + 40 X^2 + 
    4 X - 11.
\end{multline*}
The only rational points are the two rational points at infinity which we denote by
$\infty_+$ and $\infty_{-}$ (these are in fact the two cusps of $X_0(71)$). Write 
\[
	D_0=\infty_+ - \infty_{-}, \qquad D_\infty=\infty_{+}+\infty_{-}.
\]
Then,
\[
	J(\Q) \; = \; (\Z/35\Z) \cdot [ D_0],
\]
where $J=J_0(71)$.  Let $P$ be a primitive degree $6$ point on $X_0(71)$, and let $D$
be the corresponding effective irreducible degree $6$ divisor. Hence
$[D-3 D_\infty] \in J(\Q)$. It follows that
\[
	D \in \lvert D_a \rvert, \qquad D_a=a \cdot D_0+3 D_\infty, \qquad -17 \le a \le 17.
\]
We find that $\ell(D_a)$ is $4$ for $a=0$, is $3$ for $a=\pm 1$, is $2$ for $a=\pm 2$ and is $1$
for all other values of $a$. 
If $\ell(D_a) \ge 2$  then, by Corollary~\ref{cor:isolated}, 
we know that $\lvert D_a \rvert$
does not contain primitive divisors.
Thus $D \in \lvert D_a \rvert$ for $-17 \le a \le -3$ or $3 \le a \le 17$ whence $\ell(D_a)=1$.
For each of these values, $L(D_a)=\Q \cdot f_a$ where $f_a$ is a non-zero function on
$X_0(71)$. Moreover, if $D \in \lvert D_a\rvert$ then $D=D_a+\divv(f_a)$. We obtain $30$ potential possibilities for the divisor $D$.
We find that for $a=\pm 3$, the divisor $D_a+\divv(f_a)$ is reducible, and for $a=\pm 5$, $\pm 7$, $\pm 12$,
the divisor $D_a+\divv(f_a)$ is the Galois orbit of an imprimitive point. The remaining $22$ values
of $a$ yield the Galois orbit of a primitive degree $6$ point. We conclude that there
are precisely $22$ primitive degree $6$ points on $X_0(71)$ up to Galois conjugacy.

%Suppose that $C$ and $d$ satisfy the hypotheses of Theorems~\ref{thm:gonality} or \ref{thm:relative}. 
%We recall, as described in Remark~\ref{subsec:effectivity}, that if we are able to determine the representatives in \eqref{eqn:decomp} then we are able to compute all primitive degree $d$ points on $C$. We demonstrate the effectiveness of this procedure by computing all low degree primitive points on the hyperelliptic curves $X_{0}(N)$ for 
We carried out similar computations for the hyperelliptic $X_0(N)$ with $N\in\{46, 47, 59, 71\}$, and for degrees $d$
in the range $3 \le d \le \min(g,6)$ where $g$ is the genus of $X_0(N)$. 
The outcome of these computations is summarized in Table~\ref{table:X0n}. Here
we were helped by the fact 
that these values of $N$, the Mordell--Weil group $J_0(N)(\Q)$ has been computed by Bruin and Najman \cite{bruin_najman_2015}.
Furthermore, models for the curves are readily available in \texttt{Magma} 
\cite{Magma}
via the \texttt{Small Modular Curve package}.

\bigskip

In view of Corollary~\ref{cor:bielliptic}, it is natural to also consider bielliptic $X_0(N)$. 
Bars \cite{Bars} shows that $X_0(N)$ is bielliptic for precisely $41$ values of $N$. Of these,
$J_0(N)$ has analytic rank $0$ for $30$ of these values:
\begin{itemize}
\item genus $2$: $N= 22$, $26$, $28$, $50$;
\item genus $3$: $N=30$, $33$, $34$, $35$, $39$, $40$, $45$, $48$, $64$; 
\item genus $4$: $N=38$, $44$, $54$, $81$;
\item genus $5$: $N=42$, $51$, $55$, $56$, $63$, $72$, $75$;
\item genus $7$: $N=60$, $62$,  $69$;
\item genus $9$: $N=95$;
\item genus $11$: $N=94$, $119$.
\end{itemize}
Again, it is straightforward to apply Corollary~\ref{cor:bielliptic} to these
curves.  We computed all primitive points of certain low degrees on the genus
$7$
bielliptic curves $X_{0}(60)$ and $X_0(62)$.
For these two curves 
the
size of the Mordell--Weil group has been computed by Najman and Vukorepa
\cite{NV23}. We computed models for these curves
and Mordell--Weil generators using a \texttt{Magma}
package developed by Ozman and Siksek
\cite{OzmanSiksek}, Ad\v{z}aga, Keller, Michaud-Jacobs, Najman, Ozman and
Vukorepa \cite{AKMJNOV2023}, and Najman and Vukorepa \cite{NV23}. 
All computations were performed in \texttt{Magma}. We summarize
our results in Table~\ref{table:X0n}, and 
refer the reader to
\[
	\text{\url{https://github.com/MaleehaKhawaja/Primitive}} 
\]
for the supporting code as well as a description of the points.

\medskip

We also give a description of all 
effective degree $d$ divisors $D$ with $\ell(D)=1$, and
refer the reader to Table~\ref{table:X0nbreak} for this summary.

\begin{table}

\begin{tabular}{|c| c| c| c| c| c |c |} 
 \hline
	    &     &         & 
	\multicolumn{4}{c|}{Number of primitive degree}\\ 
	$N$ & $g$ & $J(\Q)$ & 
	\multicolumn{4}{c|}{$d$ points on $X_{0}(N)$}\\ 
  &  &  & $d=3$ & $d=4$ & $d=5$ & $d=6$  \\
 \hline
$46$ & $5$ & $\Z/11\Z\times\Z/22\Z$ & $2$ & $4$ & $88$ & $-$\\
\hline
$47$ & $4$ & $\Z/23\Z$ & $2$ & $12$ & $-$ & $-$ \\
\hline
$59$ & $5$ &$\Z/29\Z$ & $1$ & $2$ & $16$ & $-$ \\
\hline 
$60$ & $7$ & $\Z/4\Z\times(\Z/24\Z)^3$ & $0$ & $0$ & $120$ & $-$\\
\hline
$62$ & $7$ & $\Z/5\Z\times\Z/120\Z$ & $2$ & $0$ & $0$ & $-$\\
\hline
$71$ & $6$ & $\Z/35\Z$ & $0$ & $0$ & $0$ & $22$\\
\hline
\end{tabular}
\caption{This table gives the conclusions of our
computations of primitive points on $X_{0}(N)$
of certain low degrees $d$ and for the values of $N$ is the 
first column. Here $g$ is the genus
of $X_0(N)$, and $J(\Q)$ is in fact the structure
of the Mordell--Weil group where $J=J_0(N)$. 
The table gives the number of primitive degree $d$ points
on $X_0(N)$ up to Galois conjugacy. The symbol $-$ indicates that
	our method is inapplicable for that particular $N$ and $d$.}
\label{table:X0n}
\end{table}

\begin{table}

\begin{tabular}{|c| c| c| c| c| c| c| c| c| c| c |} 
 \hline
	$N$ & \multicolumn{2}{|c|}{$d=3$} & \multicolumn{3}{|c|}{$d=4$} & \multicolumn{2}{|c|}{$d=5$} & \multicolumn{3}{|c|}{$d=6$}\\ 
	\hline
	& $n_{3, p}$ & $n_{3, r}$ & $n_{4, p}$ & $n_{4, i}$ & $n_{4, r}$ & $n_{5, p}$ & $n_{5, r}$ &  $n_{6, p}$ & $n_{6, i}$ & $n_{6, r}$\\
 \hline
$46$ & $2$ & $20$ & $4$ & $10$ & $42$ & $88$ & $128$ & $-$ & $-$ & $-$\\
\hline
$47$ & $2$ & $2$ & $12$ & $2$ & $6$ & $-$ & $-$ & $-$ & $-$ & $-$\\ 
\hline
$59$ & $1$ & $2$ & $2$ & $0$ & $4$ & $16$ & $8$ & $-$ & $-$ & $-$ \\
\hline
$60$ & $0$ & $364$ & $0$ & $22$ & $1349$ & $120$ & $4440$ & $-$ & $-$ & $-$\\
\hline
$62$ & $2$ & $28$ & $0$ & $0$ & $58$ & $0$ & $100$ & $-$ & $-$ & $-$\\
\hline
$71$ & $0$ & $2$ & $0$ & $0$ & $2$ & $0$ & $2$ & $22$ & $6$ & $2$\\
\hline
\end{tabular}
\caption{
%Let $g$ denote the genus of the modular curve $X_{0}(N)$. 
For each pair $(N,d)$, the table gives a description of the 
effective degree $d$ divisors $D$ with $\ell(D)=1$ 
on the modular curve $X_{0}(N)$.
We denote by $n_{d, r}$ the number of such divisors
that are reducible, 
$n_{d, p}$ the number of such divisors
that are irreducible and primitive,
and 
$n_{d, i}$ the number of such divisors
that are irreducible but imprimitive.
 The symbol $-$ indicates that we did not carry out the
computation for the pair $(N, d)$.
}
\label{table:X0nbreak}
\end{table}

\section{Data Availability Statement}
The data related to this paper is available at 
\[
	\text{\url{https://github.com/MaleehaKhawaja/Primitive}} 
\]

\section{Conflict of Interest Statement}

All authors certify that they have no affiliations with or involvement in any
organization or entity with any financial interest or non-financial interest in
the subject matter or materials discussed in this paper. If however, after publication,
any organization
or entity manages to generate wealth from the theorems herein, they are exhorted
to share a generous proportion of the said wealth with the authors.

\bibliographystyle{abbrv}
\bibliography{Primitive}
\end{document}